\documentclass{amsart}
\usepackage{amsmath}
\usepackage[T1]{fontenc}
\usepackage{textcomp} 
\DeclareEncodingSubset{TS1}{hlh}{1}  

\usepackage{amscd,amssymb,epsfig, epsf,pinlabel,graphicx}
\usepackage{epic,eepic}
\usepackage[dvipsnames]{xcolor}

\usepackage[all]{xy}
\SelectTips{cm}{}
\allowdisplaybreaks

\numberwithin{equation}{subsection}

\newtheorem{propo}{Proposition}[section]
\newtheorem{corol}[propo]{Corollary}
\newtheorem{theor}[propo]{Theorem}
\newtheorem{lemma}[propo]{Lemma}

\theoremstyle{definition}

\theoremstyle{remark}

\let\oldmarginpar\marginpar
\renewcommand\marginpar[1]{\oldmarginpar{\footnotesize #1}}

\newcommand{\CC}{\mathbb{C}}

\newcommand{\Int}{\operatorname{Int}}

\newcommand{\card}{\operatorname{card}}

\begin{document}

\title[Quasi-Lie bialgebras of loops in  quasi-surfaces]{Quasi-Lie bialgebras of loops in  quasi-surfaces}

    \author[Vladimir Turaev]{Vladimir Turaev}
    \address{
    Vladimir Turaev \newline
    \indent   Department of Mathematics \newline
    \indent  Indiana University \newline
    \indent Bloomington IN47405, USA\newline
    \indent and  \newline
    \indent IRMA, Strasbourg \newline
    \indent 7 rue Rene Descartes \newline
    \indent 67084 Strasbourg, France \newline
    \indent $\mathtt{vturaev@yahoo.com}$}

\begin{abstract}     We  discuss natural operations on   loops in a quasi-surface and  show that these operations define a structure  of a quasi-Lie bialgebra
in the module  
   generated by the set of     free homotopy classes of non-contractible loops.  
\end{abstract}


\maketitle


\section {Introduction}

The module   generated by  free homotopy classes of loops in an oriented surface  carries a natural Lie bracket introduced by  W. Goldman 
\cite{Go1}, \cite{Go2} and  closely related to the   Poisson brackets on the moduli spaces of the surface,   see \cite{AB}, \cite{Wo},  \cite{FR}.
 Goldman's Lie bracket is  complemented  by a natural Lie cobracket so that together they form a Lie bialgebra,   see \cite{Tu2}.  Here we study similar operations on   loops in  more general topological spaces called quasi-surfaces.  This leads us to an algebraic notion  of a quasi-Lie  algebra, a dual notion of a quasi-Lie coalgebra, and a self-dual notion of a quasi-Lie bialgebra. We show that our operations on loops in quasi-surfaces satisfy the axioms of a quasi-Lie   bialgebra.
  
  The    idea  behind the study of quasi-surfaces   is that a    topological  space containing an oriented   surface with boundary~$\Sigma$ -   which is separated  from the rest of the space by  segments in $\partial \Sigma$  -  must  inherit   certain   features of~$\Sigma$.
Formally, a  quasi-surface~$X$  is   obtained by gluing~$\Sigma$  to an arbitrary  topological space~$Y$ along a mapping of several disjoint subsegments of  $\partial  \Sigma$,  called the gates,  to~$Y$. The space~$X$  splits as a union of the \lq\lq surface core"~$\Sigma$ and the \lq\lq singular part"~$Y$ which meet   at the gates.  
Considering loops in~$X$ and their intersections in~$\Sigma$, we obtain a bracket     in the module $M=M(X)$ 
  freely  generated by the set of     free homotopy classes of  loops  in~$X$, see    \cite{Tu3}. This bracket  is skew-symmetric but 
 may  not satisfy the Jacobi identity.  Here we compute its Jacobiator      in terms  of    operations on loops  associated with   the gates.    Similarly, considering  self-intersections of loops, we obtain    a skew-symmetric   cobracket  in~$M$.  We compute its co-Jacobiator    and   coboundary    in terms of     the gates. These computations show  that our operations  on loops induce a quasi-Lie bialgebra structure in the quotient of the module~$M$ by the homotopy class of contractible loops in~$X$.
  For    loops      in the surface core~$\Sigma$   of~$X$,  we recover   the standard    Lie bialgebra of  loops in~$\Sigma$.

A major source of quasi-surfaces are finite families of disjoint segments in (ordinary) surfaces.
Consider such a family~$C$ in an oriented surface~$\Sigma$. Assume that $C\cap \partial \Sigma =\partial C\neq \emptyset$ and that~$C$ splits~$\Sigma$ into two surfaces   meeting  at~$C$.   Taking one of them as the   surface core,  the other one as the singular part, and the components of~$C$ as the gates, we turn~$\Sigma$ into a  quasi-surface.  Our operations on loops   yield  then a quasi-Lie bialgebra structure in $M(\Sigma)$ depending on~$C$.
Further examples of  quasi-surfaces    may be produced by collapsing   several segments in $\partial 
\Sigma$ into a single point which plays  the role of the  singular part.
 

Other known operations on  loops in  surfaces,  see  \cite{KK}, \cite{KK2}, \cite{MTnew},   can also  be generalized to  quasi-surfaces. The author plans to discuss these generalizations  elsewhere.

The paper starts with the definitions of  quasi-Lie algebras, coalgebras, and bialgebras (Section~\ref{brbrr}). Then we discuss   gate operations on loops in  topological spaces (Section~\ref{A topological example})  and formulate our main theorems (Section~\ref{A topological example2}). The rest of the paper is devoted to the proof of these theorems.


This work was  supported by the NSF grant DMS-1664358.

\section{Brackets, cobrackets, and bi-endomorphisms}\label{brbrr}
 
 \subsection{Brackets}\label{topolsssett}\label{topddolsett}
Throughout the paper we fix a commutative   ring~$R$. By a module we mean an $R$-module. For a module~$M$ and an integer $m\geq 1$, we let $M^{  m}$  be  the direct product  of~$m$ copies of~$M$.   An  \emph{$m$-bracket} in~$M$ is a  map $ M^{ m} \to M$ which is linear in all~$m$ variables. 
A  bracket  in~$M$   is     \emph{fully symmetric}  if it is invariant under all permutations of the variables.  A bracket $ \mu:M^{ m} \to M$ is \emph{cyclically symmetric} if $\mu(x_1,..., x_m)=\mu (x_2,..., x_m, x_1)$ for all $x_1,..., x_m\in M$. 
Given  a 2-bracket $\mu:M^{ 2} \to M$, its \emph{transpose} $\mu^t$ is the 2-bracket  in~$M$  defined by     $\mu^t(x, y)=\mu(y,x)$    for all $x, y \in M$.  The \emph{Jacobiator} $J_\mu $   of~$\mu$ is the cyclically symmetric 3-bracket   in~$M$ 
  defined by
$$ J_\mu (x,y,z)= \mu(\mu(x,y),z ) +  \mu(\mu(y,z), x)+  \mu(\mu(z,x), y)$$
for all  $x,y,z\in M$. Note  the identity  
$$ J_{\mu^t} (x,y,z)=\mu(z, \mu(y,x))+ \mu(x , \mu(z,y)) +\mu(y,\mu(x,z))  .  $$
A 2-bracket~$\mu$  in~$M$ is \emph{symmetric} (respectively, \emph{skew-symmetric})  if $\mu^t=\mu $ (respectively, $\mu^t=-\mu $). 
In both cases $J_{\mu^t}=J_\mu$.  
 
 \subsection{Quasi-Lie algebras}\label{Bracketsssy modbbvules}   
   A \emph{quasi-Lie algebra}  is  a  module~$M$  endowed with a  skew-symmetric 2-bracket $[-,-]$  and   a cyclically symmetric 3-bracket
$[-,-,-]$  such that the Jacobiator  of $[-,-]$ is obtained by antisymmetrization of    $[-,-,-]$, i.e.,  \begin{equation}\label{JJJacoee}   J_{[-,-]}(x,y,z)= [x,y,z]- [ z,y,x] \end{equation} 
for all $x,y,z \in M$.  Such a pair $([-,-], [-,-,-])$ is called  a \emph{quasi-Lie pair of brackets} in~$M$. We recover the usual Lie algebras   when  $[-,-,-]=0$.

 The following  lemma is our main tool producing quasi-Lie  algebras.
  
  \begin{lemma}\label{sDDfffDtrbucte}    
  For any module~$M$ and any bilinear form $M^{ 2} \to M, (x,y)\mapsto x\bullet y$,  the 2-bracket  $[x,y]=x\bullet  y-y\bullet  x$ and the 3-bracket 
$$ [x,y,z]=J_\bullet (x,y,z)+ J_{\bullet^t} (x,y,z)  $$
$$=(x\bullet y)\bullet z +(y\bullet z)\bullet x+ (z\bullet x)\bullet  y + z\bullet (y\bullet x) +  x\bullet (z\bullet y) + y\bullet (x\bullet z)  $$
form a  quasi-Lie pair of brackets  in~$M$. 
   \end{lemma}
  
\begin{proof}   That the 2-bracket $[-,-]$ is  skew-symmetric  and   the 3-bracket
$[-,-,-]$ is cyclically symmetric is clear. Direct computations show that
$$[[x,y],z] =[x \bullet  y-y\bullet  x, z]= (x\bullet  y) \bullet z-(y\bullet x)\bullet z -z\bullet  (x\bullet y)+z\bullet  (y\bullet  x),$$
$$[[y,z], x]=[y\bullet  z-z\bullet  y, x]= (y\bullet  z) \bullet  x- (z\bullet  y) \bullet x -x\bullet  (y\bullet z)+ x\bullet (z\bullet  y), $$
$$ [[z,x], y]= [z\bullet  x-x\bullet z, y]= (z\bullet x)\bullet y-(x\bullet z)\bullet y -y\bullet (z\bullet  x)+y\bullet (x\bullet z).$$
Summing up, we get 
$$ J_{[-,-]}(x,y,z)=  J_\bullet (x,y,z)- J_\bullet (z,y,x) - J_{\bullet^t} (z, y,x) +J_{\bullet^t} (x,y,z) $$
$$ = [x,y,z]- [ z,y,x]. $$
 \end{proof}

 \subsection{Remarks} 1. Given a quasi-Lie pair of brackets $([-,-], [-,-,-])$  in a module~$M$ and a fully symmetric 3-bracket~$u $  in~$M$,  the pair $([-,-], [-,-,-] - u)$ is also a quasi-Lie pair of brackets in~$M$. For example, taking in Lemma~\ref{sDDfffDtrbucte}    
$$u(x,y,z)=J_{\bullet^t} (x,y,z) +J_{\bullet^t} (z,y,x)$$
we deduce that  the 2-bracket  $[x,y]=x\bullet  y-y\bullet  x$ and the 3-bracket carrying any triple $(x,y,z)\in M^3$ to 
$$  (x\bullet y)\bullet z - x\bullet (y\bullet z)  +(y\bullet z)\bullet x -  y\bullet (z\bullet x) + (z\bullet x)\bullet  y     -  z\bullet (x\bullet y)$$
form a  quasi-Lie pair of brackets in~$M$.  This  shows that   quasi-Lie algebras naturally arise in the study of non-associative multiplications. 

2. 
A  quasi-Lie pair of brackets    in a module~$M $  gives rise to  a fully symmetric 3-bracket $s:M^{  3}\to M$    by
$$ s(x,y,z)= [x,y,z]  +[z,y,x]  = 2 [x,y,z] - J_{[-,-]}(x,y,z)$$
for any $x,y,z\in M$.  
If 2 is invertible in~$R$, then   we can recover the 3-bracket $[-,-,-]$ from $[-,-]$  and~$s$.    
 This yields a bijective correspondence between  quasi-Lie pairs of brackets in~$M$  and  pairs
(a skew-symmetric 2-bracket in~$M$, a fully symmetric 3-bracket in~$M$).

\subsection{Cobrackets and quasi-Lie coalgebras}\label{+algprelcobr} To discuss the dual notions of  cobrackets  and quasi-Lie coalgebras we use the language of tensor products.
For a module~$M$ and an integer $m\geq 1$, we let $M^{\otimes m}$  be  the tensor product over~$R$  of~$m$ copies of~$M$.  It is clear that $m$-brackets in~$M$ bijectively correspond to linear maps  $M^{\otimes m}  \to M$. An   \emph{$m$-cobracket} in~$M$ is a linear map $M\to  M^{\otimes m}  $. A cobracket $ \nu:M\to M^{\otimes m} $ is \emph{cyclically symmetric} if $\nu=Q_m \circ \nu$ where $ Q_m$ is   the linear automorphism of $M^{\otimes m}$ defined by  $Q_m(x\otimes y)= y\otimes x$ for all $x\in M, y\in M^{\otimes (m-1)}$.
 For a 2-cobracket $\nu  : M\to M^{\otimes 2} $, we denote by $\nu^2$  the 3-cobracket $(\nu \otimes {\text {id}}_M) \circ \nu: M\to M^{\otimes 3} $.  The \emph{co-Jacobiator} $j_\nu $   of~$\nu$ is the cyclically symmetric 3-cobracket  $$ j_\nu = ( I +Q+Q^2) \circ \nu^2 : M\to M^{\otimes 3}$$
  where $I={\text {id}}_{M^{\otimes 3}}$ and  $Q=Q_3$.
A 2-cobracket $\nu $      is \emph{symmetric} (respectively, \emph{skew-symmetric})  if $P\circ\nu =  \nu$ (respectively, $P\circ\nu=-   \nu$) where  $P=Q_2$ is     the permutation automorphism of $ M^{\otimes 2}$
 defined by  $P(x\otimes y)= y\otimes x$ for all $x,y\in M $.

A \emph{quasi-Lie coalgebra}  is  a  module~$M$  carrying  a  skew-symmetric 2-cobracket~$\nu$  and   a cyclically symmetric 3-cobracket
$\gamma: M \to M^{\otimes 3}$  such that  
$$ j_\nu=  E \circ  \gamma: M \to M^{\otimes 3} $$
where  $E$ is the endomorphism of  $M^{\otimes 3}$  defined by  \begin{equation}\label{HHH}E(x\otimes y \otimes z)=x\otimes y \otimes z- z\otimes y \otimes x \end{equation}
for all $x,y,z \in M$.   The pair $(\nu, \gamma)$ is called then a \emph{quasi-Lie pair of cobrackets} in~$M$. We recover the usual Lie coalgebras   for $\gamma=0$.

\subsection{Bi-endomorphisms and quasi-Lie bialgebras}\label{+algprelcobssar} By  a \emph{bi-endomorphism} of a module~$M$ we mean a linear endomorphism of the module $M^{\otimes 2}$. For instance, the permutation automorphism~$P$ of $ M^{\otimes 2}$, see Section~\ref{+algprelcobr},    and  
  $\overline P={\text {id}}_{M^{\otimes 2}} -P $ are bi-endomorphisms of~$M$. 
  We say that a bi-endomorphism~$\zeta$   of~$M$    is \emph{equivariant} if  $P\zeta  =\zeta P$.   Examples:  $P$ and~$ \overline P$  are equivariant;   for each bi-endomorphism~$\zeta$ of~$M$, the  bi-endomorphism  $  \zeta^{eq}=\zeta+ P   \zeta   P $  is  equivariant.
We say that a bi-endomorphism~$\zeta$   of~$M$    is \emph{skew-symmetric} if  $P\zeta  =\zeta P=-\zeta$.  
    Examples:  $ \overline P$   is skew-symmetric; for each bi-endomorphism~$\zeta$ of~$M$, the  composition   $ \delta (\zeta)= \overline P   \zeta   \overline P $  is skew-symmetric. 
    Note that  \begin{equation}\label{equis} \delta (\zeta^{eq})=2\,  \delta (\zeta). \end{equation}
    
  
  

  
    A  bracket $[-,-]$  and  a   2-cobracket $\nu$ in~$M$ determine a bi-endomorphism $\partial \nu$ of~$M$,  the  \emph{coboundary  of~$\nu$}.  Note  that any    $z\in M$ yields a linear map $ad_z: M^{\otimes 2}\to M^{\otimes 2}$   carrying $ x \otimes y$ to $[z,x] \otimes y+ x \otimes [z,y]$ for any $x,y \in M$. By definition, $$\partial \nu (x \otimes y)= \nu ( [x,y])  - ad_x( \nu(y)) +   ad_y( \nu(x)) \in M^{\otimes 2}.$$
  If  $[-,-]$ and~$\nu$ are  skew-symmetric, then so is $\partial \nu$. We have  $\partial  \nu \circ P =- \partial  \nu$   because      $$\partial \nu (y \otimes x)= \nu ( [y,x])  - ad_y( \nu(x)) +   ad_x( \nu(y)) $$
$$=-\nu ( [x,y])  + ad_x( \nu(y)) -   ad_y( \nu(x))=-\partial \nu (x \otimes y) $$
for any    $x,y \in M$.
Also  $  P\circ \partial \nu  =- \partial  \nu$   because 
$$ P(\partial \nu (x \otimes y))= P \big (\nu ( [x,y]) - ad_x( \nu(y))+ ad_y( \nu(x)) \big ) $$ $$ =-\nu ( [x,y])  +ad_x(  \nu(y) )  -ad_y(  \nu(x) ) =  -\partial \nu (x \otimes y).$$


We define a quasi-Lie bialgebra to be  a module  endowed with a quasi-Lie pair of brackets $([-,-] , [-,-,-])$, a  quasi-Lie pair of cobrackets   $(\nu, \gamma)$, and an equivariant   bi-endomor\-phism~$ \zeta $
such that  $\partial \nu=\delta (\zeta)$. We leave it to the reader to verify that this  notion    of a quasi-Lie bialgebra is self-dual. For $[-,-,-]=0, \gamma=0,  \zeta=0$, we recover the usual   Lie bialgebras.


 \section{Gates and loops}\label{A topological example}
 
 We  define   gates in   topological spaces and   derive from the gates certain  algebraic operations on   homotopy classes of loops.
 

  \subsection{Preliminaries}\label{prel1}
  A \emph{loop} in a   topological space~$X$  is a continuous map $a:S^1\to X$ where the circle $S^1=\{p\in \CC\, \vert \, \vert p \vert=1\}$ is oriented counterclockwise. The point $a(1)\in X$ is   the \emph{base point} of~$a$.  For   $p\in   S^1$,  we let $a_p: S^1\to X$ be the  loop  obtained as the composition of~$a$ with the rotation $S^1\to S^1$ carrying $1\in S^1$ to~$p$.  
  

 Two loops $a,b:S^1 \to X$ are \emph{freely homotopic} if there is a continuous map $F:S^1\times [0,1]\to X$  such that $F(p,0)=a(p)$ and $F(p,1)=b(p)$ for all $p\in S^1$.  Such a map~$F$ is   a \emph{homotopy} between~$a$ and~$b$.
 
We   let $   {L} (X)$ denote    the set of free   homotopy classes of   loops in~$X$ and   let  $  M(X) $    be   the  free module with basis $ { L}    (X)$.  
For  a   loop~$a$ in~$X$, we let $\langle a \rangle \in  {L} (X) \subset  M(X) $ be the free homotopy class of~$a$. Set $\langle a \rangle_0=\langle a \rangle \in    {M}(X)$   if~$a$ is non-contractible and     $\langle a \rangle_0=0 \in {M}(X)$  if~$a$ is contractible.

 \subsection{Gates}\label{topolsett}\label{coordinateddd2w-}  
    A \emph{cylinder neighborhood} of a  subset~$C$ of a topological space~$X$  is   a  pair consisting of a closed set $U\subset X$  with $C \subset  {\rm {Int}} (U)$ and     a homeomorphism  $ U\approx C\times [-1,1]$    carrying  ${\rm {Int}} (U)$ onto  $C\times (-1, 1)$ and carrying each point $c\in C   $ to $(c,0)  $.    A  \emph{gate} in~$X$ is  a closed path-connected  subspace $C \subset X$ endowed with a cylinder neighborhood   and such that all loops in~$C$ are contractible in~$X$.  We  will identify the   cylinder neighborhood in question with $C\times [-1,1]$ via the given homeomorphism.     


 For a gate $C \subset X$, consider     the  map $H:X \to S^1 $  carrying  the complement of  $ C\times (-1,1) $  in~$X$ to $-1\in S^1$ and  carrying $C\times \{t\}  $ to ${\rm{exp} }(\pi i t)\in S^1$ for all $t\in [-1,1]$. 
We say that a  loop $a:S^1 \to X$ is \emph{transversal} to~$C $  if   the map $Ha: S^1 \to S^1$   is transversal to $1  \in S^1$. Then   the set $a^{-1}(C) =(Ha)^{-1}(1)$  is finite.   For each  $p\in a^{-1}(C)$,  we   define the {\emph {crossing sign}} $\varepsilon_p(a) $: if  at~$p$ the loop~$a$    goes  from $C\times [-1,0) $ to $C \times (0,1]$  then   $\varepsilon_p(a)=+1$, otherwise,  $\varepsilon_p(a)=-1$. The integer
$$ a\cdot C= \sum_{p\in a^{-1}(C)} \varepsilon_p(a)   $$
is   the   algebraic  intersection number of~$a$ and~$C$.   
Note  that the formula $a\mapsto a \cdot C$ defines a homomorphism $H_1(X)\to \mathbb {Z}$.      An example of a gate is provided by a   simply connected proper codimension~1    submanifold~$C$ of a manifold together with a suitable homeomorphism of a closed neighborhood of~$C$   onto $C\times [-1,1]$. In this case, $a\cdot C$ is the usual intersection number of a  loop~$a$ with~$C$.  
  
  \subsection{Gate brackets} \label{Gate brackets}  
We   derive  from a   gate $C \subset  X$  a sequence of brackets $\{\mu^m_C\}_{m\geq 1}$ in the module $ M =M(X) $. Fix a  point $\star \in C$. For a loop $a$ in~$X$ with   $a(1) \in  C$ we let $\widetilde a$ be the loop  based in~$\star$ and obtained from $a$ by conjugation along a   path   in~$C$   from~$\star$ to $a(1)$. The loop $\widetilde a$      represents an element of   $\pi_1(X, \star)$   independent of the choice of the latter path.
 Consider now $m \geq 1$   loops $a_1,..., a_m$ in~$X$   transversal to~$C$.   For   any $i=1,...,m$ and   $p_i\in a_i^{-1}(C) \subset S^1$,  consider the reparametrization $a_{i,p_i}=(a_i)_{p_i}$ of $a_i$  based at $a_i(p_i)\in C$, see  Section~\ref{prel1}.  Consider the associated loop $\widetilde a_{i,p_i} $  based at  $\star$.  Set
  \begin{equation}\label{twotwo3aa} \mu^m_C (  a_1  , \ldots ,   a_m ) =
 \sum_{p_1\in a_1^{-1}(C), \ldots, p_m\in a_m^{-1}(C)} \prod_{i=1}^m \varepsilon_{p_i}(a_i)\,  \langle \prod_{i=1}^m\widetilde  a_{i,p_i} \rangle \in M  .   \end{equation}

 \begin{lemma}\label{strleleucte}  $\mu^m_C (  a_1  , \ldots ,   a_m )  $ depends only on the free homotopy classes of  the loops $a_1,..., a_m$.
 \end{lemma}

\begin{proof} For any pair of  freely homotopic loops in~$ X$ transversal to~$C$,  there is a free homotopy between these loops  which splits as a product of several homotopies of the following two  types  (and   inverse homotopies):  (i)  deformations   in the class of loops transversal to~$C$ and (ii) deformations  pushing a branch of the loop lying in $X \setminus C$   across~$C$ and creating   two new crossings of the loop with~$C$. It is clear that homotopies of   $a_1,..., a_m$ of   type (i) do not change $\mu^m_C (  a_1  , \ldots ,   a_m )$. A homotopy of $a_k$ of  type  (ii)   with $k\in \{1,..., m\}$ creates two   crossings $p, p'$ such that  $\varepsilon_{p}(a_k)=-\varepsilon_{p'}(a_k)$ and  $\widetilde  a_{k,p}=\widetilde  a_{k,p'}$. As a consequence, the terms of the   sum \eqref{twotwo3aa} arising from $p_k=p$ and   $p_k=p'$ cancel   each other, so that the sum is preserved. It is  clear that $\mu^m_C (  a_1  , \ldots ,   a_m ) $ does not depend on the choice of the point $\star \in C$.
 \end{proof} 
 
Using the product structure in   the cylinder neighborhood   of~$C$,     we easily observe that  each loop in~$X$ is freely homotopic to  a loop   transversal to~$C$. Therefore Lemma~\ref{strleleucte}  yields a map $$({L}(X))^m \to M, \,\, (\langle a_1 \rangle,..., \langle a_m \rangle)\mapsto  \mu^m_C (  a_1  , \ldots ,   a_m ) .$$ This map extends uniquely to an  $m$-bracket $\mu^m_C$  in~$  M$. Since the loops $\prod_{i=1}^m\widetilde  a_{i,p_i}$ and $
  (\prod_{i=2}^m\widetilde  a_{i,p_i})\,  \widetilde  a_{1,p_1}  $  are freely homotopic,  this bracket    is cyclically symmetric for all~$m$. Clearly,  $\mu^1_C (  \langle a \rangle)= (a\cdot C) \langle a \rangle$ for any loop~$a$ in~$X$ transversal to~$C$.
  

   \subsection{Gate cobrackets}\label{CCCcobrackets} 
For a  gate $C$ in a topological space~$X$  and for any integer $m\geq 1$, we define  an  $m$-cobracket     in the module $M=M(X)$. The idea is to  consider  all splittings of a loop $a:S^1\to X$    as a product of~$m$ paths
with endpoints in~$C$ and   to sum up the associated elements  of $M^{\otimes m}$. 
To this end, for any   distinct points $p_1,p_2 \in S^1$, we consider the  path in~$X$ obtained by restricting~$a$ to the arc in $S^1$ which starts in~$p_1$ and goes counterclockwise  to~$p_2$.  If  $a(p_1), a(p_2 )
 \in C$, then the product of this path with a  path    from $a(p_2)$ to~$a(p_1) $  in~$C$ is a loop in~$X$ based at  $a(p_1)$. This loop is denoted  
$a_{p_1,p_2}$. Since~$C$ is path-connected and  all loops in~$C$ are contractible in~$X$, the  loop $a_{p_1,p_2}$  is well defined up to  homotopy.  
An \emph{$m$-sequence}    of the loop~$a $  is  a  sequence  of~$m$ distinct points
$p_1,..., p_m \in  a^{-1}(C)$ such that    moving along $S^1$ counterclockwise, we meet   $p_1,..., p_m$ in this cyclic order. 
 The set of all     $m$-sequences of~$a$ is denoted  $S_m(a)$. 
It  is nonempty if and only if $m\leq \card (a^{-1}(C))$. If~$a$ is  transversal  to~$C $, then 
the set $S_m(a)$ is finite and we set 
  \begin{equation}\label{tweeotwo3} \gamma_{C,m} (a)= \sum_{(p_1,..., p_m)  \in S_m(a)} \,  \prod_{i=1}^m\varepsilon_{p_i}(a) \, \bigotimes_{i=1}^m \, \langle a_{p_i, p_{i+1}} \rangle_0  \in M^{\otimes m}   \end{equation}  where   $p_{m+1}=p_1$.   Since all 
    cyclic permutations of  $m$-sequences of~$a$      are $m$-sequences of~$a$, the vector  $\gamma_{C,m} (a) $ is invariant under   cyclic permutations of  $M^{\otimes m}$.     In particular,      $\gamma_{C,1} (a)=(a\cdot C) \, \langle  a \rangle_0$  and  $$
 \gamma_{C,2} (a)= \sum_{(p_1,  p_2 )\in S_2(a) } \,    \varepsilon_{p_1}(a) \, \varepsilon_{p_2}(a)\,  \langle a_{p_1, p_2} \rangle_0  \otimes  \langle a_{p_2, p_1} \rangle_0 $$
where  $S_2(a)$ is the set of all ordered pairs of distinct elements of $a^{-1}(C) \subset S^1$. 
 
  \begin{lemma}\label{svvvvtrbucte}  For all  $m \geq 1$, Formula \eqref{tweeotwo3} defines a  map $L (X)\to M^{\otimes m}$.
\end{lemma}
  
\begin{proof}   We need only to prove that if two   loops $a,a'$  transversal to~$C$ are freely homotopic in~$X$, then 
 $\gamma_{C,m} (a) =\gamma_{C,m} (a')$. As in the proof of Lemma~\ref{strleleucte}, it is enough to consider a homotopy pushing  a   branch  of the loop  across~$C$ and creating   two transversal  crossings   with opposite crossing signs. The   contributions to $\gamma_{C,m} $ of  the  $m$-sequences  containg neither of these two  points are the same before and after the deformation. The   contributions to $\gamma_{C,m} $ of  the  $m$-sequences  containing exactly one of these new crossings cancel each other.   The    $m$-sequences  containing both new crossings  contribute zero to  $\gamma_{C,m}$ because at least  one of the corresponding   loops $a_{p_i, p_{i+1}}$ is contractible in~$X$.   This implies our claim. \end{proof}

The  map $L (X) \to M^{\otimes m}$  from  Lemma~\ref{svvvvtrbucte} extends  uniquely  to a   linear map $ \gamma_{C,m} :  M\to M^{\otimes m}$. This map   is  an $m$-cobracket  in $M$. 

\subsection{Gate bi-endomorphisms} \label{Gate bi-endos} 
We   derive  from any   gate $C \subset  X$  a bi-endomor\-phism $  \zeta_C$ of  the module $ M =M(X) $.  Note  that for any     loops $a,b:S^1\to X$  transversal to~$C$ and  for any points $p\in a^{-1}(C)$, $q\in b^{-1}(C)$,  we  can   multiply the    loops $a_p, b_q$  connecting their base points $a(p), b(q)\in C$ by a path  in~$C$.   The product loop is denoted $a_p \circ_C \,  b_q$. Its free homotopy   class  $\langle a_p \circ_C \, b_q \rangle$ does not depend on the choice of the path  in~$C$  connecting  the base points. 
Set  
 $\vert a \vert_C   = \card (a^{-1}(C))$ and 
\begin{equation}\label{zetaminus}
 \zeta (a,b)  =  \vert a \vert_C \, (b\cdot C ) \, \langle a  \rangle_0 \otimes \langle b  \rangle_0 \end{equation} $$ + \,\, 2 \, \sum_{(p_1,  p_2 )\in S_2(a) , q \in b^{-1}(C)}  \,  \varepsilon_{p_1}(a) \, \varepsilon_{p_2}(a)\, \varepsilon_q(b) \,   \langle  a_{p_1, p_2} \rangle_0  \otimes  \langle a_{p_2, p_1}  \circ_C \,  b_q  \rangle_0  .  $$
 Here $a_{p_1, p_2} , a_{p_2, p_1}  $ are the loops as in Section~\ref{CCCcobrackets} 
    based respectively at $p_1, p_2$.


\begin{lemma}\label{svvvvtssarbucte}      $  \zeta (a,b)  \in M^{\otimes 2} $   is invariant  under free homotopies  of~$a$ and~$b$ in~$X$.
\end{lemma}
  
\begin{proof}   We first  show that $ \zeta(a,b)= \zeta(a,b')$ for any loop~$b'$ transversal to~$C$ and freely homotopic to~$b$.    As in   Lemma~\ref{svvvvtrbucte}, it is enough to show that $  \zeta(a,b)$   is preserved under   any deformation $b \mapsto b'$ of~$b$ pushing  a branch   across~$C$ and creating   two crossing points $q_+, q_-$ so that $(b')^{-1}(C)= b^{-1}(C) \amalg  \{q_+, q_-\}$ and $\varepsilon_{q_\pm} (b')=\pm 1$.   Then   for any pair $(p_1,  p_2 )\in S_2(a)$, the  contributions   of the  triples  $p_1,  p_2   , q_+$  and $p_1,  p_2   , q_-$  to $  \zeta(a,b')$    are opposite and  cancel  each other.        Also, $b' \cdot C= b\cdot C$ and $\langle b'  \rangle_0 =\langle b  \rangle_0 $. Therefore  $ \zeta(a,b')= \zeta (a,b)$. 

We next show that $   {\zeta}(a,b) =  {\zeta}(a',b)$ for any loop~$a'$ transversal to~$C$ and   freely    homotopic to~$a$. As above, it suffices to consider the case where $a'$ is obtained by pushing  a branch of~$a$    across~$C$   creating  two   crossing points    $p_+, p_-$ so that  $(a')^{-1}(C)= a^{-1}(C) \amalg \{p_+, p_-\}$  and $\varepsilon_{p_\pm} (a')=\pm 1$.  Let~$s$ be the sum over   $p_1,p_2,q$ appearing in~\eqref{zetaminus}. Let~$s'$ be the similar  sum with~$a$ replaced by  $a'$. For any  $p\in a^{-1}(C) $, $q\in b^{-1}(C)$, the contributions of the triples $p_+, p, q$ and $p_-, p, q$ to $s'$  cancel each other.
The same holds for the triples $p, p_+, q$ and $p, p_-, q$.  To compute the contributions of the  triples $p_+, p_-, q$ and $p_-, p_+, q$   to
$s'$, we  distinguish  two cases. In the first case, the deformation $a\mapsto a'$  pushes a branch of~$a$ \lq\lq up\rq\rq, that is from $C\times [-1, 0)$  to $C\times (0,1]$. 
Then  the loop $a'_{p_+, p_-}$ is contractible and  the loop $a'_{p_-, p_+}$ is freely homotopic to~$a$.    So,     the triples $p_+, p_-, q$ and $p_-, p_+, q$ contribute to
$s'$ respectively $0$ and $-\varepsilon_q(b) \,   \langle    a  \rangle_0  \otimes  \langle b  \rangle_0 $.  In the second case, the deformation $a\mapsto a'$  pushes a branch of~$a$ \lq\lq down\rq\rq, that is  from $C\times (0,1]$  to  $C\times [-1, 0)$. 
Then  the loop $a'_{p_+, p_-}$ is freely homotopic to~$a$,  the loop $a'_{p_-, p_+}$ is contractible, and the triples $p_+, p_-, q$ and $p_-, p_+, q$ contribute to
$s'$ respectively    $-\varepsilon_q(b) \,   \langle   a  \rangle_0  \otimes  \langle b  \rangle_0 $ and $0$. 
In both cases, 
$$s'=s- \sum_{q\in b^{-1}(C)} \varepsilon_q(b)\,   \langle    a  \rangle_0  \otimes  \langle b  \rangle_0 = s- (b \cdot C)\,   \langle    a  \rangle_0  \otimes  \langle b  \rangle_0 .$$ Multiplying this equality by 2 and adding to  the obvious formula 
$$\vert a' \vert_C \, (b\cdot C ) \, \langle a'  \rangle_0 \otimes \langle b  \rangle_0=(\vert a \vert_C +2)\, (b\cdot C ) \, \langle a  \rangle_0 \otimes \langle b  \rangle_0$$
 we obtain that  $ {\zeta} (a', b)= {\zeta} (a, b)$.
  \end{proof} 


 By Lemma~\ref{svvvvtssarbucte},  the formula $(a,b) \mapsto {\zeta} (a, b)$ defines a map
$L(X) \times L(X)  \to M^{\otimes 2}$.   We let $\zeta_C:  M^{\otimes 2}\to M^{\otimes 2}    $ be the linear extension  of this  map.

  \section{Quasi-surfaces and main theorems}\label{A topological example2}
  
  We recall quasi-surfaces from \cite{Tu3} and state our main theorems.
  
\subsection{Quasi-surfaces}\label{Terminology}\label{notat}    By  a   \emph{surface} we mean  a  smooth oriented  2-dimensional manifold  with   boundary.  
A \emph{quasi-surface} is a path-connected  topological space~$X$ obtained by gluing a surface~$\Sigma$ to a topological space~$Y$ along a continuous map   $ f: \alpha  \to Y$ where $ \alpha \subset \partial \Sigma $ is a    union of a finite number ($\geq 1$) of    disjoint  closed segments  in $\partial \Sigma$. 
Note that  then  $Y\subset X$ and   $X\setminus Y =\Sigma\setminus \alpha$.    
We   fix a closed    neighborhood  of~$\alpha$ in~$\Sigma$ and 
  identify   it  with  $\alpha \times [-1,1]$ so that 
  $$\alpha=\alpha \times \{-1\} \quad  \quad {\text {and}} \quad \quad  \partial \Sigma  \cap (\alpha \times [-1,1])=\alpha   \cup (\partial \alpha \times [-1,1]).$$
  The  surface
$$\Sigma'=\Sigma \setminus  (\alpha \times [-1,0)) \subset \Sigma \setminus \alpha \subset X   $$  is  a  copy of $\Sigma$    embedded in~$X$. 
We    provide~$\Sigma'$  with the   orientation induced  from that of~$\Sigma$. We call~$Y$ the \emph{singular part} of~$X$, call~$\Sigma$ the \emph{surface core} of~$X$, and     call~$\Sigma'$  the \emph{reduced surface core} of~$X$.


Set  $\pi_0=\pi_0(\alpha)$.   For    $k\in  \pi_0 $,   we   let   $ \alpha_k $ be  the corresponding   segment component of $\alpha \times \{0\} \subset \partial \Sigma' \subset X$. 
  It is clear that   $\alpha_k$    is a gate of~$X$ in the sense of Section~\ref{topolsett}.
The gates $\{\alpha_k\}_{k\in \pi_0}$      separate $
\Sigma'\subset X$ from the rest of~$X$.    For any loop $a: S^1 \to X$ and  $k\in \pi_0$,   we  set $a\cap \alpha_k=a(S^1) \cap \alpha_k $.

 
We say that a loop $a: S^1 \to X$ is    \emph{generic} if  (i)  all  branches of~$a$   in~$\Sigma'$   are smooth immersions meeting $\partial \Sigma'$ transversely   at a finite set of points which all lie  in the  interior  of the gates, and (ii)  all self-intersections of~$a$ in $\Sigma'$ are double transversal intersections which all lie  in $\Int(\Sigma')=\Sigma' \setminus \partial \Sigma'$.     A generic loop~$a$     traverses  any    point of a gate   $  \alpha_k$   at most once. So,  the restriction of the map $a:S^1\to X$ to $a^{-1}(\alpha_k) \subset S^1$ is a bijection onto the set  $ a \cap \alpha_k$. In  this context, we adjust   notation of Section~\ref{A topological example} and systematically use the letter~$p$ for   points in  $a  \cap \alpha_k$ rather than for their preimages under~$a$.
Accordingly, the crossing sign $\varepsilon_p(a)$ at $p\in a \cap \alpha_k$      is  $+1$ if~$a$ goes at~$p$  from $X\setminus \Sigma'$  to $\Int(\Sigma')$   and is   $-1$ otherwise. 
 
 More generally, a finite family of   loops in~$X$    is    \emph{generic}  if    these loops  are generic and all  their intersections   in~$\Sigma'$ are double transversal intersections    in $\Int(\Sigma')$. In particular,   these loops  do not meet at the gates. 
Using cylinder neighborhoods of the gates, it is easy to see that  any finite family of loops   in~$X$ can  be transformed into a generic  family  by a small deformation.

 We keep the objects  $X, \Sigma, \Sigma', \alpha, \pi_0$ for the rest of the paper.


\subsection{The   brackets}\label{The form bullet} 
 We   recall    the     2-bracket  in  the module $M=M(X)$ introduced in~\cite{Tu3}.   For a loop $a: S^1 \to X$ and a  point $r \in X$    traversed by~$a$ exactly once,
 we let $a_r$ be the   loop  which starts at~$r$ and goes along~$a$ until coming back to~$r$. 
  For   any  loops $a,b$ in~$X$ set    $$a\cap b= a(S^1) \cap b(S^1)\cap \Sigma'   .$$  If  $a,b$ is a generic  pair  of loops  then the set $a\cap b \subset   \Int(\Sigma')$ is finite and  each point $r \in a\cap b$ is traversed by~$a$ and~$b$ only once so that we can consider the loops $a_r, b_r$ based at~$r$.    Set $\varepsilon_r  (a,b) =  1$ if  the tangent vectors of~$a$ and~$b$ at~$r$ form a    positive basis in the tangent space of~$\Sigma'$ at~$r$ and   set $\varepsilon_r   (a,b)=-1$ otherwise. The sum $\sum_{r\in a\cap b} \varepsilon_r  (a,b) \langle a_r b_r \rangle\in M$ may be viewed as the \lq\lq homotopy intersection'' of $a, b$.  Generally speaking, this sum depends on the choice of  $a,b$ in their free homotopy classes. 
We therefore add further terms involving  the gates. 
 
 We start with terminology.     By a   \emph{gate orientation}  of~$X$    we   mean   an orientation of all gates $\{\alpha_k\}_k$  of~$X$. 
    Gate orientations of~$X$  canonically correspond to orientations of the 1-manifold $\alpha \subset \partial \Sigma$.
Pick a   gate orientation~$\omega$ of~$X$. For    $k\in \pi_0$,    set  $\varepsilon(\omega, k)=+1$ if the $\omega$-orientation of~$\alpha_k$ is compatible with the orientation of~$\Sigma'$, i.e.,  if the pair (a $\omega$-positive tangent vector  of  $\alpha_k \subset \partial \Sigma'$, a   vector  directed inside~$\Sigma'$) is  positively oriented in~$\Sigma'$. Otherwise,  set  $\varepsilon(\omega, k)=-1$.   For      points $p,q \in \alpha_k$,  we  
 write  $p<_\omega q$
if $p\neq q$ and  the $\omega$-orientation  of $\alpha_k$   leads from~$p$ to~$q$. 

Any  generic pair of  loops $a,b$ in~$X$ determines a  finite  set of triples   
$$T(a,b)=\{(k,p,q) \, \vert \, k\in \pi_0, p\in a\cap \alpha_k, q\in b\cap \alpha_k \}.$$
These  triples are called     \emph{chords} of the pair  $(a,b)$.      For any such chord  $(k,p,q)$  we have $p \neq q$  since  generic loops do not meet at the gates.    Also,  we can   multiply the loops $a_p, b_q$  based at $p,q$ using an arbitrary  path   connecting  $p,q$ in $\alpha_k$. The product  loop determines an  element of $L (X)$   denoted $\langle a_p b_q \rangle$. 
Clearly,  if  $(k, p,q) $ is a chord of $(a,b)$ then $(k, q,p)$ is a chord of $(b,a)$ and $\langle a_p b_q \rangle=\langle  b_q  a_p \rangle$. Finally, set
$$T_\omega (a,b)=\{(k,p,q) \in T(a,b) \, \vert \, q<_\omega p  \} \subset T(a,b)  .$$


\begin{lemma}\label{1aeee++} (i) There is a unique 2-bracket $[-,-]_{X, \omega}$ in the module $M=M(X)$ such that for any generic pair of loops $a,b$ in~$X$, we have 
$$
[\langle a \rangle, \langle b \rangle]_{X,\omega} =  \sum_{r\in a\cap b} \varepsilon_r  (a,b) \langle a_r b_r \rangle +\sum_{(k,p,q)\in T_\omega (a,b) }\,    \varepsilon(\omega, k)   \,\varepsilon_p(a)  \,\varepsilon_q(b) \langle a_pb_q \rangle. $$

(ii) The skew-symmetric bracket $[-, -]_X$ in~$M$ defined by
$$
[x,y]_X  =   
 [x,y]_{X,\omega}  -  [y,x]_{X,\omega}  $$
 for all $x,y \in M$  does not depend on the choice of~$\omega$.
\end{lemma}

  Claim (i) is Lemma 4.1 of  \cite{Tu3} and   Claim (ii) is Theorem 4.2 of  \cite{Tu3}.  
Both   brackets $[-,-]_{X,\omega}$ and $[-,-]_{X}$    generalize Goldman's  bracket (\cite{Go1}, \cite{Go2}): 
 the value of   $[-,-]_{X,\omega}$  (respectively,  $[-,-]_{X}$)  on any pair of  free homotopy  classes  of loops in $\Sigma'\subset X$    is equal to  their  Goldman's    bracket    (respectively,  twice this bracket). For all~$\omega$, the bracket  $2\, [-,-]_{X,\omega}$ can be computed from $[-,-]_{X}$ and  the gate 2-brackets  of~$X$, see \cite{Tu3}, Remark 4.5.1.

We now compute     the Jacobiator of    $[-,-]_X$. 
By Section~\ref{Gate brackets}, each gate $\alpha_k$ of~$X$ determines  cyclically symmetric  brackets $\{\mu^m_k=\mu^m_{\alpha_k}\}_{m\geq 1}$  in~$M$.  The sum  
 $$\mu^m=\sum_{k \in \pi_0} \mu^{m}_{k}: M^{m} \to M$$ 
 is a  cyclically symmetric  bracket in~$M$ called  the \emph{total gate $m$-bracket} of~$X$.

\begin{theor}\label{MAIN}   The module $M=M(X)$ endowed with the 2-bracket $[-,-]=[-,-]_X$ and the total gate 
 3-bracket $\mu=\mu^3$  is a quasi-Lie algebra. 
\end{theor}

Theorem~\ref{MAIN} may be rephrased by saying that  
$$ J_{[-,-]_X}(x,y,z)= \mu(x,y,z) - \mu (z,y, x)     $$
   for any $x,y,z\in M$. Theorem~\ref{MAIN} is proved in Section~\ref{section3ghgh}.

 \subsection{The  cobrackets}\label{section2-}
\label{koko}\label{kokoff}\label{koko89}\label{moves}     We   recall  from ~\cite{Tu3}   the   intersection cobracket  in    $M=M(X)$. 
  Consider a generic loop~$a$ in~$X$.   Denote by $\#a$ the set of self-intersections  of~$a$   in $\Sigma'$. This set is finite and  lies in $\Int(\Sigma')$.  The loop~$a$ crosses each point  $r\in \# a$ twice; we let $v^1_r, v^2_r$ be the tangent vectors of~$a$ at~$r$ numerated so    that the pair $(v^1_r, v^2_r)$ is positively oriented. For $i=1,2$,   let $a^i_r$ be the loop   starting in~$r$ and going  along~$a$   in the direction of the vector $v^i_r$    until the first return to~$r$. Up to parametrization, $a=a^1_ra^2_r$ is the product of the  loops $a^1_r, a^2_r$ based at~$r$.  
As in Section~\ref{CCCcobrackets}, for  any  distinct points
$p_1, p_2 \in a \cap \alpha_k$ with  $k\in \pi_0$ we have the loop $a_{p_1, p_2}$   which goes from $p_1$  to $p_2$ along~$a$ and   then goes back  to~$p_1 $  along the gate~$\alpha_k$. 
  For a gate orientation~$\omega$ of~$X$,  we    consider the set of ordered triples 
  $$T_\omega(a)  =   \{(k,p_1,p_2) \, \vert \,  k\in \pi_0,p_1 \in a \cap \alpha_k , p_2 \in a \cap \alpha_k ,   \,  p_1<_\omega p_2\}.$$
  We  call  a   triple   $ (k,p_1, p_2) \in  T_\omega (a) $ a  \emph{chord of~$a$ with endpoints} $p_1, p_2$.

  \begin{lemma}\label{svvvvtreebucte}  (i) There is a unique 2-cobracket $\nu_{X, \omega}$ in the module $M=M(X)$ such that for any generic   loop~$a $ in~$X$, we have 
$$
\nu_{X, \omega} (\langle a \rangle)=\sum_{r\in \# a} (\langle a^1_r \rangle_0 \otimes \langle a^2_r \rangle_0  -\langle a^2_r \rangle_0  \otimes \langle a^1_r \rangle_0) $$
$$
 +\sum_{(k,p_1, p_2)\in T_\omega(a)} \, \varepsilon(\omega, k) \,  \varepsilon_{p_1}(a) \, \varepsilon_{p_2}(a) \,    \langle  a_{{p_2},{p_1}}\rangle_0 \otimes \langle  a_{{p_1},{p_2}} \rangle_0 .
 $$

(ii) Let~$P$ be
 the linear  automorphism  of  $M \otimes M $ carrying $x\otimes y$ to $y\otimes   x$   for all $x,y \in M$.   The skew-symmetric cobracket 
$$
\nu_X  =   
 \nu_{X,\omega} - P \nu_{X,\omega}: M \to M \otimes M $$ 
does not depend on the choice of~$\omega$.
\end{lemma}

  Claim (i) is Lemma 5.1 of  \cite{Tu3} and   Claim (ii) is Theorem 5.2 of  \cite{Tu3}.  
  For loops  in~$\Sigma'$, the cobracket $\nu_{X,\omega}$ coincides with  the    cobracket introduced in  \cite{Tu2} and the cobracket~$\nu_X $  is twice the    one   in  \cite{Tu2}. 
   For all~$\omega$, the cobracket $2 \, \nu_{X,\omega}$ can be recovered from $\nu_{X}$ and the gate 2-cobrackets  of~$X$. 

We now compute     the co-Jacobiator of   $\nu_X$. 
By Section~\ref{A topological example}, each gate $\alpha_k$ of~$X$ determines cyclically symmetric  cobrackets $\{\gamma_{  \alpha_k,m}\}_{m\geq 1}$ in   $M $.
The  map  
$$\gamma^m=\sum_{k\in \pi_0} \gamma_{  \alpha_k,m} :M\to M^{\otimes m} $$
is a cyclically symmetric $m$-cobracket in $M$ called the \emph{total gate $m$-cobracket}
  of~$X$.
Though we shall not need it, note  that  $\gamma^1=0$.

\begin{theor}\label{MAIN2}   The module $M=M(X)$ endowed with the 2-cobracket $\nu=\nu_X$ and the total gate 
 3-cobracket $\gamma=\gamma^3$  is a quasi-Lie coalgebra. 
\end{theor}

 Theorem~\ref{MAIN2}  is proved in Section~\ref{section3ghgh22}.
Next, we  compute  the coboundary $\partial \nu$ of $\nu=\nu_X$ via  the gates (at least in the case where $1/2\in R$).
  By Section~\ref{Gate bi-endos}, each gate $\alpha_k$ of~$X$ yields  a bi-endomorpism $\zeta_{\alpha_k}$ of~$M$.  The   map 
 \begin{equation}\label{zzet} \zeta=\sum_{k \in \pi_0} \zeta_{\alpha_k}: M^{\otimes 2} \to M^{\otimes 2} \end{equation}
 is  called  the \emph{total gate bi-endomorphism} of the module $M=M(X)$. 
 
 
 \begin{theor}\label{MAIN3} Let $e\in L(X) \subset M$ be the   homotopy class of contractible loops.  Then $$\delta (\zeta) =2 \,\partial \nu \, \,\,  {\rm  {mod}} (Re\otimes M+M\otimes Re).$$
 \end{theor}
 
 Theorem~\ref{MAIN3}  is proved in Section~\ref{coorcompadfT0}. 
 This theorem  suggests to consider  the    quotient module $M_\circ=M/Re$. We can represent  $e\in   M$  by an embedded circle   in a small disc in $\Int(\Sigma')$. This easily implies that
$$\mu (Re \times M\times M)=\mu (  M\times Re \times  M)
=\mu (  M \times  M \times Re ) =0,$$
$$[Re, M]= [M,Re]= 0, \,\, \nu(Re)=0, \,\, \gamma(Re)=0, \,\, \zeta  (Re \times M )=\zeta (  M\times Re)=0.$$ 
Denoting the projection  $M\to M_\circ$  by~$\psi$, we deduce the existence and uniqueness of   maps $[-,-]_\circ, \mu_\circ, \nu_\circ, \gamma_\circ, \zeta_\circ$  such that the following five diagrams commute: 
$$
\xymatrix{
M \times M \ar[rr]^-{[-,-]} \ar[d]_{\psi \times \psi} & & M\ar[d]^-{\psi}\\
{M_\circ \times M_\circ} \ar[rr]^-{[-,-]_\circ}&  &M_\circ
}, \quad  \xymatrix{
M \times M \times M \ar[rr]^-{\mu} \ar[d]_{\psi \times \psi \times \psi} & & M\ar[d]^-{\psi}\\
{M_\circ \times M_\circ \times M_\circ} \ar[rr]^-{\mu_\circ}&  &M_\circ
},
$$
$$
\xymatrix{
M   \ar[r]^-{\nu} \ar[d]_{  \psi}  & M^{\otimes 2}\ar[d]^-{\psi^{\otimes 2}}\\
{M_\circ } \ar[r]^-{\nu_\circ}  &M_\circ^{\otimes 2}
}, \quad  \xymatrix{
M   \ar[r]^-{\gamma} \ar[d]_{  \psi}  & M^{\otimes 3}\ar[d]^-{\psi^{\otimes 3}}\\
{M_\circ} \ar[r]^-{\gamma_\circ}  &M_\circ^{\otimes 3}
}, \quad  \xymatrix{
M^{\otimes 2}   \ar[r]^-{\zeta} \ar[d]_{  \psi^{\otimes 2}}  & M^{\otimes 2}\ar[d]^-{\psi^{\otimes 2}}\\
{M_\circ^{\otimes 2} } \ar[r]^-{\zeta_\circ}  &M_\circ^{\otimes 2}
}.
$$

Theorems \ref{MAIN}, \ref{MAIN2}, \ref{MAIN3} and Formula~\eqref{equis} imply the following:
 
\begin{corol}\label{MAIN4}   If $1/2\in R$, then the module $M_\circ$    with  brackets $[-,-]_\circ, \mu_\circ$, cobrackets $\nu_\circ, \gamma_\circ$, and   bi-endomorphism $(1/4) (\zeta_\circ)^{eq}$
  is   a quasi-Lie bialgebra.    
\end{corol}

\subsection{Remark} M. Chas \cite{Ch} proved that the Lie bialgebras of loops on (ordinary) surfaces are  involutive.
It would be interesting to extend this result to the quasi-Lie bialgebra  of Corollary~\ref{MAIN4}.


 \section{Proof of Theorem~\ref{MAIN} }\label{section3ghgh}

  \subsection{Preliminaries}\label{section3loops+} 
A finite family  of loops in the quasi-surface~$X$  is said to be \emph{simple} if   these loops  meet the gates of~$X$ transversely and have no crossings or  self-crossings    in the reduced surface core~$\Sigma'\subset X$.  In particular, these loops do not meet at the gates. A simple family of loops is necessarily generic.

 \begin{lemma}\label{Loopsiemnn}  Any finite family of loops in~$X$  can  be deformed in~$X$ into a simple family of loops.
\end{lemma}

\begin{proof} Consider first a   single loop in~$X$. Since~$X$ is path-connected and contains a gate, we can deform our  loop  into a  generic loop~$a$ which  meets a gate.
 If the set $\# a$ of double points of~$a$ in $\Int(\Sigma')$ is empty, then we are done. 
 Otherwise, pick a  point $r \in \# a$. Starting   at $r=r_0$ and moving along~$a$   we     meet several double  points $r_1,..., r_n  \in \# a$ with $n\geq 0$ and then  come to a  point $p\in a\cap \alpha_k$ in  a certain   gate $\alpha_k$.   The  segment   of~$a$   running from  $r_n$ to~$p$  is  embedded    in $\Sigma'$  and meets  $  \# a$ only at  $r_n$. Denote this segment by~$s$ and let~$t$  be   the branch of~$a$ transversal to~$s$ at   $r_n$.   Push~$t$ towards~$p$ along~$s$   keeping~$t$   transversal  to~$s$ and eventually push~$t$ across  $\alpha_k$ at~$p$. This deformation of~$a$ increases    $\card (a \cap \alpha_k)$ by 2 and   decreases ${\rm {card}} (\# a)$ by 1. Continuing by induction, we    deform~$a$ into a generic loop without   self-intersections  in~$\Sigma'$.  
 If the original  family of loops contains two or more  loops, then we first deform it into a generic family of loops which  all  meet some gates. Then, as above,  pushing branches at crossings and self-crossings across the gates, we obtain a simple family of loops. \end{proof} 
 

  \subsection{Proof of the theorem}\label{section3ghvvgh}   
 Fix an arbitrary gate orientation~$\omega$ of~$X$ and let~$\overline \omega$   be  the  gate orientation of~$X$   opposite to~$\omega$ on all  gates. To shorten our formulas,  we set $x\bullet y  =[x,y]_{X, \omega}$ for any $x,y \in M={M}(X)$.
If $x,y \in {L}(X)\subset M$ are represented by a generic pair of loops $a,b$, then 
    \begin{equation}\label{eenewidyvKKs}     x \bullet   y  =  \sum_{r\in a\cap b} \varepsilon_r  (a,b) \langle a_r b_r \rangle +\sum_{(k,p,q)\in T_\omega (a,b) }\,    \varepsilon(\omega, k)   \,\varepsilon_p(a)  \,\varepsilon_q(b) \langle a_pb_q \rangle. \end{equation} Note that the inclusion $(k,p,q)\in T_\omega (a,b)$ holds if and only if $(k,q,p)\in T_{\overline \omega} (b,a)$. This  and the obvious  identities $$\langle a_r b_r \rangle=  \langle b_r a_r \rangle, \quad  \varepsilon_r(a,b)=- \varepsilon_r(b,a), \quad \varepsilon ( \overline \omega, k)= - \varepsilon (   \omega, k)$$ imply  that \begin{equation}\label{chorisdfbulle}   
 [x,y]_{X, \overline\omega} =- [y,x]_{X,   \omega} =-y \bullet x  .   \end{equation}
By the bilinearity of the brackets, the same  holds  for all $x,y\in M$. 

For arbitrary $x,y,z\in M$,  set 
   \begin{equation}\label{newidyvKK}     u_\omega(x,y,z)= (x \bullet   y) \bullet  z  +(y \bullet   z) \bullet  x+(z \bullet   x) \bullet  y.     \end{equation}
By \eqref{chorisdfbulle},
\begin{equation}\label{idyvg3n+} u_{\overline \omega} (x,y,z)=[[x,y]_{X, \overline \omega}, z]_{X, \overline \omega}+ [[y,z]_{X, \overline \omega}, x]_{X, \overline \omega}+[[z,x]_{X, \overline \omega}, y]_{X, \overline \omega}
\end{equation}$$=
z \bullet  ( y  \bullet x)  + x \bullet  ( z \bullet  y)+ y \bullet   (x \bullet  z)  . $$
We now apply Lemma~\ref{sDDfffDtrbucte}    to  the
  bilinear form $M^{ 2} \to M, (x,y)\mapsto x \bullet y$. The 2-bracket $[-,-]$   in this lemma is  the bracket $[-,-]_X$. Formulas~\eqref{newidyvKK}  and \eqref{idyvg3n+} show  that the 3-bracket   $[-,-,-]$ in Lemma~\ref{sDDfffDtrbucte}  is equal to $u _\omega +u_{\overline \omega}$. By~\eqref{JJJacoee}, to prove the theorem it suffices   to show  that    for all $x,y,z\in M$, we have  \begin{equation}\label{key}
[x,y,z] - [z, y,x]= \mu^3(x,y,z) - \mu^3 ( z, y, x) .  \end{equation} 

 Since both sides  of~\eqref{key} are linear in $x,y,z$, it suffices to handle  the case   $x,y,z\in {L}(X)  \subset M$.
By Lemma~\ref{Loopsiemnn},    we can represent the triple $x,y,z$ by a simple  triple of loops  $a,b,c$  in~$X$. Then  Formula~\eqref{eenewidyvKKs} simplifies to
\begin{equation}\label{chorieee++vj}  x \bullet y      
= \sum_{(k,p,q) \in  T_\omega (a,b) }\,    \varepsilon(\omega, k)   \,\varepsilon_p(a)  \,\varepsilon_q(b) \, \langle a_p  b_q  \rangle \in M \end{equation}
 where   $a_p, b_q $ are   loops reparametrizing $a,b$ and based respectively in $p, q \in \alpha_k$. Here
  $a_p b_q= a_p \ell  b_q \ell^{-1} $  for a path  $\ell=\ell (p,q)$    in   the gate $\alpha_k$ from~$p$ to~$q$. We deform the loop $a_p b_q $    by  slightly pushing its subpaths~$\ell$ and $  \ell^{- 1}$    into $ X\setminus   \Sigma'$.  (The endpoints $p,q$ of these subpaths  are pushed  into $ X\setminus   \Sigma'$ along $a,b$, respectively.) The resulting  loop is denoted   by $a\circ_{p,q}b$ or by $b\circ_{q,p} a$. Thus,   \begin{equation}\label{monoee} 
 x \bullet y   =  \sum_{(k,p,q) \in T_\omega (a,b) }\,   \varepsilon(\omega, k)   \,\varepsilon_p(a)  \,\varepsilon_q(b) \, \langle a\circ_{p,q}b \rangle. \end{equation}
 Note that the loop $a\circ_{p,q}b$  is simple;   moreover,  the pair of loops   $a\circ_{p,q}b, \, c$   is simple.
Applying~\eqref{chorieee++vj} to this pair, we get 
$$
\langle a\circ_{p,q}b \rangle  \bullet z   =  \sum_{(l,s,t) \in  T_\omega (a\circ_{p,q}b, c) }\,    \varepsilon(\omega, l)   
\,\varepsilon_s(a\circ_{p,q}b)  \,\varepsilon_t(c)\,  \langle (a\circ_{p,q}b)_s c_t \rangle. $$
For each  $l\in \pi_0$,   the set     $(a\circ_{p,q} b)\cap  \alpha_l$ is a disjoint  union of   the sets $a\cap \alpha_l$ and $b\cap \alpha_l$. 
   Therefore for  every triple  $(k,p,q) \in T_\omega (a,b)$, we have  $$T_\omega (a\circ_{p,q}b, c) =T_\omega (a , c)  \sqcup  T_\omega ( b, c)  $$  and   $$ \langle a\circ_{p,q}b \rangle  \bullet z  =\lambda_\omega(k,p,q) + \rho_\omega(k,p,q)$$ where
\begin{equation}\label{varp}\lambda_\omega(k,p,q)=  \sum_{(l,s,t)\in T_\omega (a,c) }\,    \varepsilon(\omega, l)   
\,\varepsilon_s(a)  \,\varepsilon_t(c) \, \langle (a\circ_{p,q}b)_s c_t \rangle \in M, \end{equation}
\begin{equation}\label{psps}\rho_\omega (k,p,q)=  \sum_{(l,s,t) \in T_\omega (b,c) }\,    \varepsilon(\omega, l)   
\,\varepsilon_s(b)  \,\varepsilon_t(c) \, \langle (a\circ_{p,q}b)_s c_t \rangle \in M .\end{equation}
Combining   with \eqref{monoee}  we obtain   
\begin{equation}\label{monoeg} (x \bullet y )\bullet z= \sum_{(k,p,q) \in T_\omega (a,b) }\,   \varepsilon(\omega, k)   \,\varepsilon_p(a)  \,\varepsilon_q(b)\, \big (\lambda_\omega(k,p,q) +\rho_\omega(k,p,q) \big ). \end{equation}

   
   We now compute the right-hand sides of Formulas \eqref{varp}-\eqref{monoeg}.
The free homotopy class   $\langle (a\circ_{p,q}b)_s c_t \rangle$  in~\eqref{varp} is represented by the loop $b\circ_{q,p} a_{s,t}  \circ c$ obtained by grafting~$b$ and~$c$ to~$a$  using  a  path  in $\alpha_k$ from $p\in a\cap \alpha_k$ to $ q\in b\cap \alpha_k$ and a path  in $\alpha_l$  from $s\in a\cap \alpha_l$ to $t\in c\cap \alpha_l$. So,
\begin{equation}\label{monoeg+}\lambda_\omega(k,p,q)=  \sum_{(l,s,t)\in T_\omega (a,c)}\,    \varepsilon(\omega, l)   
\,\varepsilon_s(a)  \,\varepsilon_t(c) \, \langle b\circ_{q,p} a_{s,t}  \circ c \rangle. \end{equation}
To give  a more precise description of the loop $b\circ_{q,p} a_{s,t}  \circ c$, we distinguish  two classes of triples    $(l,s,t)\in T_\omega (a,c)$ determined by whether or not $s=p$.

Class 1: $s\neq p$ so that  the points $p ,q ,s ,t $ are  distinct (possibly, $l=k$). Then the loop $b\circ_{q,p} a_{s,t}  \circ c$  goes     along  $\alpha_k$  from~$p$ to~$q$,  along the full loop~$b$ from~$q$ to~$q$,   along $\alpha_k$ from~$q$ to~$p$,    along~$a$ from~$p$ to~$s$,    along   $\alpha_l$ from~$s$ to~$t$,   along the full  loop~$c$ from~$t$ to~$t$,   along~$\alpha_l$ from~$t$  to~$s$, and  finally along~$a$ from~$s$  to~$p$. 

Class 2: $s= p$ so that  $l=k$ and    $p ,q ,t \in \alpha_k$ are three distinct points.
If $\varepsilon_p(a)=+1$, then the loop $b\circ_{q,p} a_{s,t}  \circ c$     goes    along $\alpha_k$ from~$p$  to~$q$,  along the full loop~$b$ from~$q$ to~$q$,     along $\alpha_k$ from~$q$ to~$t$,  along the full  loop~$c$ from~$t$ to~$t$, along $\alpha_k$  from~$t$ to~$p$,  and finally  along the full loop~$a$ from~$p$   to~$p$.
 If $\varepsilon_p(a)=-1$, then the loop $b\circ_{q,p} a_{s,t}  \circ c$    goes    along $\alpha_k$ from~$p$  to~$q$,  along the full loop~$b$ from~$q$ to~$q$,     along $\alpha_k$ from~$q$ to~$p$,   along the full  loop~$a$ from~$p$ to~$p$,   along $\alpha_k$ from~$p$ to~$t$,  along the full loop~$c$ from~$t$ to~$t$,  and finally
   along~$\alpha_k$ back  to~$p$.

Similarly, the free homotopy class   $\langle (a\circ_{p,q}b)_s c_t \rangle$ in~\eqref{psps} is represented by the loop $a\circ_{p,q} b_{s,t}  \circ c$ obtained by grafting~$a$ and~$c$ to~$b$  via a path  in $\alpha_k$ from $p\in a\cap \alpha_k$ to $ q\in b\cap \alpha_k$ and  a path  in $\alpha_l$  from $s\in b\cap \alpha_l$ to $t\in c\cap \alpha_l$. A precise description  of this loop  involves  two classes of triples $(l,s,t) \in T_\omega ( b,c)$ determined by  whether or not $s=q$; we leave the details  to the reader. Thus,   
\begin{equation}\label{monoeg++}\rho_\omega (k,p,q)=  \sum_{(l,s,t) \in T_\omega ( b,c) }\,    \varepsilon(\omega, l)   
\,\varepsilon_s(b)  \,\varepsilon_t(c) \langle a\circ_{p,q} b_{s,t}  \circ c \rangle. \end{equation} 

   Substituting these expansions of  $\lambda_\omega, \rho_\omega$ in \eqref{monoeg}, we obtain an expansion of  $(x \bullet y )\bullet z$ as a sum. We call the summands   determined by     pairs  
$(k,p,q)\in  T_\omega (a,b)$, $ (l,s,t)\in  T_\omega (a,c)$ with $s\neq p$ the \emph{left 4-terms}. The   summands   determined by    pairs   $(k,p,q)\in  T_\omega (a,b)$, $ (l,s,t)\in  T_\omega (a,c)$ with $s=p$ (and $ l=k$) are  called  the \emph{left 3-terms}. Similarly,   the summands of $(x\bullet y) \bullet z$  determined by     pairs  
$(k,p,q)\in  T_\omega (a,b)$, $ (l,s,t)\in  T_\omega (b,c)$ with $s\neq q$ are called the \emph{right 4-terms}. The   summands of $(x\bullet y) \bullet z$ determined by    pairs   
$(k,p,q)\in  T_\omega (a,b)$, $ (l,s,t)\in  T_\omega (b,c)$ with $s=q$ (and $ l=k$) are  called  the \emph{right 3-terms}. Thus,  $(x\bullet y) \bullet z$  is  a sum of  (left and right) 3-terms  and  4-terms.
Further, using \eqref{newidyvKK}, we   expand   $u_\omega (x,y,z)$  as a sum of    3-terms  and  4-terms. Combining  with a parallel expansion of  $ u_{\overline \omega} (x,y,z) $   we get an expansion of   $ [x,y,z]=u_\omega(x,y,z)+  u_{\overline \omega} (x,y,z) $ as a sum of  3-terms  and  4-terms. The  total contribution   of  the    3-terms (respectively, 4-terms)  to $  [x,y,z]$ is denoted by $  P_3 (a,b,c)$ (respectively, $P_4 (a,b,c)$). Thus,  $$  [x,y,z] = P_3 (a,b,c)+  P_4 (a,b,c).$$

To proceed, we  simplify our notation.   Note   that   each point  $p,q,s,t$   in a 4-term or a 
3-term    is traversed by exactly  one of the loops $a,b,c$. We will write $\varepsilon_p, \varepsilon_q, \varepsilon_s, \varepsilon_t$ for the corresponding signs $\pm 1$. For example, $\varepsilon_p=\varepsilon_p (a), \varepsilon_q=\varepsilon_q(b)$, etc. Also,   set
  $\varepsilon(\omega, k, l) = \varepsilon(\omega, k)  \, \varepsilon(\omega, l) $.   In this notation, the       contribution of  the left 4-terms  to     $(x\bullet y) \bullet z$
is equal to 
$$\lambda_\omega^{a,b,c}= \sum_{ \begin{array}[b]{r}
     { (k,p,q) \in T_\omega (a,b) }\\
      {(l,s,t)\in T_\omega (a,c), s\neq p }
    \end{array} } \,  \varepsilon(\omega, k, l) \,\varepsilon_p \, \varepsilon_q \, \varepsilon_s \,\varepsilon_t \,  \langle b\circ_{q,p} a_{s,t}  \circ c \rangle.$$
    The  contribution of  the right 4-terms  to    $(x\bullet y) \bullet z$
is equal to 
$$\rho_\omega^{a,b,c}= \sum_{ \begin{array}[b]{r}
     { (k,p,q) \in T_\omega (a,b) }\\
      {(l,s,t)\in T_\omega (b,c), s\neq q }
    \end{array} } \,   \varepsilon(\omega, k, l)  \,\varepsilon_p \, \varepsilon_q \, \varepsilon_s \,\varepsilon_t \, \langle a\circ_{p,q} b_{s,t}  \circ c \rangle. $$
Applying to the formula for  $\lambda_\omega^{a,b,c}$  the permutations
$$  a\mapsto b \mapsto c \mapsto a, \quad k \mapsto l \mapsto k, \quad p\mapsto s \mapsto q \mapsto t \mapsto p$$ we get 
  $$\lambda_\omega^{b,c,a}= \sum_{ \begin{array}[b]{r}
     { (l,s,t) \in  T_\omega (b, c) }\\
      {(k,q,p)\in T_\omega (b,a), q\neq s}
    \end{array} }\,   \varepsilon(\omega, k, l)  \,\varepsilon_p \, \varepsilon_q \, \varepsilon_s \,\varepsilon_t \, \,  \langle c\circ_{t,s} b_{q,p} \circ a \rangle.$$
Set  $$\Delta_\omega^{a,b,c}=\rho_\omega^{a,b,c}+\lambda_\omega^{b,c,a}\, \,\,\, \,\,{\text {and}} \, \,\,\,\,\,  \Delta^{a,b,c} =\Delta_\omega^{a,b,c} + \Delta_{\overline \omega}^{a,b,c}.$$
   The summands in the expansions of  $\rho_\omega^{a,b,c} $ and $\lambda_\omega^{b,c,a}$  are given  by the same formula as   follows from the equality   \begin{equation}\label{derg} \langle a\circ_{p,q} b_{s,t}  \circ c \rangle= \langle c \circ_{t,s} b_{q,p} \circ a \rangle  \end{equation}  
     for  $q\neq s$. The summation in these two expansions    goes over complementary sets  of triples $ (k,p,q)$ as the inclusion    $(k,q,p)\in T_\omega (b,a)$   holds if and only if $ (k,p,q) \in T(a,b) \setminus T_\omega (a,b) $.  
  Therefore 
  $$ \Delta_\omega^{a,b,c}=  \sum_{ \begin{array}[b]{r}
     { (k,p,q) \in T (a,b) }\\
      {(l,s,t)\in T_\omega (b,c), s\neq q }
    \end{array} }\,   \varepsilon(\omega, k, l)  \,\varepsilon_p \, \varepsilon_q \, \varepsilon_s \,\varepsilon_t \,\, \langle a\circ_{p,q} b_{s,t}  \circ c \rangle. $$
Since  $\varepsilon(\omega, k, l) = \varepsilon(\overline \omega, k, l) $, we deduce    that
$$ \Delta^{a,b,c}  = \sum_{ \begin{array}[b]{r}
     { (k,p,q) \in T (a,b) }\\
      {(l,s,t)\in T  (b,c), s\neq q }
    \end{array} }\,   \varepsilon(\omega, k, l)  \,\varepsilon_p \, \varepsilon_q \, \varepsilon_s \,\varepsilon_t \,\,  \langle a\circ_{p,q} b_{s,t}  \circ c \rangle. $$
    Applying to this formula    the label permutations
$$  a\mapsto  c \mapsto a, \quad k \mapsto l \mapsto k, \quad p\mapsto t \mapsto p, \quad q \mapsto s \mapsto q$$ we get 
$$ \Delta^{c,b,a}  =\sum_{ \begin{array}[b]{r}
     {  (l,t,s)\in T (c,b)}\\
      { (k,q,p) \in T(b,a)  , s\neq q }
    \end{array} }\,   \varepsilon(\omega, k, l)  \,\varepsilon_p \, \varepsilon_q \, \varepsilon_s \,\varepsilon_t \,\,  \langle c \circ_{t,s} b_{q,p}  \circ a \rangle.$$
Comparing with the expression for $ \Delta^{a,b,c} $ above and using  \eqref{derg}, we get 
    \begin{equation}\label{ders} \Delta^{a,b,c}  = \Delta^{c,b,a} .\end{equation}

It is clear that the  total  contribution  of  the 4-terms to 
   $u_\omega (x,y,z)$ is equal to $$\lambda_\omega^{a,b,c}+\rho_\omega^{a,b,c}+ \lambda_\omega^{b,c,a} + \rho_\omega^{b,c,a}+ \lambda_\omega^{c,a,b}  + \rho_\omega^{c,a,b}= \Delta_\omega^{a,b,c}+ \Delta_\omega^{b,c,a}+ \Delta_\omega^{c,a,b}.$$
Therefore
  $$P_4 (a,b,c)= \Delta^{a,b,c} +  \Delta^{b,c,a}+  \Delta^{c,a,b} .$$ 
  Formula \eqref{ders} implies that  $P_4 (a,b,c)= P_4(c,b,a)$.
 Therefore all 4-terms cancel out in 
 $   [x,y,z]-[z,y,x]$ and  
\begin{equation}\label{simplif}  [x,y,z]-[z,y,x]=P_3 (a,b,c)-P_3 (c,b,a) .  \end{equation} 
  It remains   to check that
 \begin{equation}\label{00-}  P_3 (a,b,c)-P_3(c,b,a)=\mu^3(x,y,z)-\mu^3(z, y,x). \end{equation}
 
    We will use the  function~$\eta$   on the set $\{\pm 1\}=\{-1, 1\}$ defined  by  $\eta (1)=1 $  and  $\eta (-1)=0 $. For any triple 
$  \varepsilon, \varepsilon', \varepsilon''\in \{\pm  1\}$   set  
\begin{equation}\label{dd} \vert \varepsilon, \varepsilon', \varepsilon'' \vert =  \varepsilon  \varepsilon' \eta (\varepsilon'') + 
 \varepsilon  \varepsilon'' \eta (\varepsilon') + \varepsilon'  \varepsilon''  \eta (\varepsilon) -\varepsilon  \varepsilon'  \varepsilon''\in {\mathbb Z}.\end{equation} Clearly,  the integer $\vert \varepsilon, \varepsilon', \varepsilon'' \vert$   is    invariant under  permutations of   
  $  \varepsilon, \varepsilon', \varepsilon''$. 
   
 Observe that each 3-term   in $P_3(a,b,c)$  is associated with a certain  $k=l \in \pi_0$ and   a triple $p \in a\cap \alpha_k, q\in b\cap \alpha_k, t\in c\cap \alpha_k$.
 Set
 \begin{equation}\label{zz} \begin{vmatrix}
a&b&c\\
p&q&t\\
\end{vmatrix}= \vert \varepsilon_p, \varepsilon_q, \varepsilon_t \vert \,  \big ( \langle a_p  b_q c_t  \rangle+
 \langle   c_t  b_q a_p \rangle \big ) -\varepsilon_p  \varepsilon_q  \varepsilon_t  \langle   c_t  b_q a_p \rangle \in M  \end{equation} where   $ a_p b_q    c_t  $ is the product of the loops $a_p, b_q, c_t$   formed by  connecting  their base points $p,q,t$  by  arbitrary paths in $\alpha_k$. The loop  $   c_t b_q a_p$ is defined similarly.

 
 
  Let $P_{3,k} (a,b,c)$ be the sum of the 3-terms in $[x,y,z]$ associated with    $k\in \pi_0$.  
We prove below  that  
\begin{equation}\label{00}  P_{3,k} (a,b,c)   =\sum_{p,q,t}\,\,   \begin{vmatrix}
a&b&c\\
p&q&t\\
\end{vmatrix}    . \end{equation}
Here and   below   $  p,q,t$ run respectively over  the sets  $ a\cap \alpha_k, b\cap \alpha_k, c\cap \alpha_k$.   Observe   that~\eqref{00}  implies \eqref{00-}. Indeed, 
permuting $a \leftrightarrow c$, $p\leftrightarrow t$ in~\eqref{00} we get 
$$  P_{3,k} (c,b,a)  =\sum_{p,q,t}\,\,   \begin{vmatrix}
c&b&a\\
t&q&p\\
\end{vmatrix}   =\sum_{p,q,t}\,\,  \begin{vmatrix}
a&b&c\\
p&q&t\\
\end{vmatrix}+\varepsilon_p  \varepsilon_q  \varepsilon_t  \langle   c_t  b_q a_p \rangle -\varepsilon_p  \varepsilon_q  \varepsilon_t  \langle    a_p b_q c_t  \rangle   $$
$$  = P_{3,k} (a,b,c) +\mu^3_{k} ( z,y,x ) -\mu^3_{k} (x,y,z).$$
 Therefore
 $$ P_{3,k} (a,b,c)- P_{3,k} (c,b,a) =\mu^3_{k} (x,y,z)- \mu^3_{k} ( z,y,x ).$$
Summing up  over 
 all $k\in \pi_0 $ and using the obvious equality
   $$P_3 (a,b,c)=\sum_{k\in \pi_0} P_{3,k} (a,b,c)$$ we  get \eqref{00-}.
 
We now prove~\eqref{00}.
Observe first that the formulas above    simplify for the left  3-terms:   $\varepsilon(\omega, k,l)=\varepsilon(\omega, k) \, \varepsilon(\omega,   l)=+1$ as $k=l$ and $\varepsilon_s \varepsilon_p=+1$ as $s=p$.      Thus, the  contribution to $(x\bullet y) \bullet z$  of  the left 3-terms   with the given~$k$
is   equal to  $${\mathcal L}_{\omega,k}^{a,b,c}=  \sum_{  
     { q<_\omega p, \, t<_\omega p}
  } \,  \varepsilon_q \, \varepsilon_t    \, \langle b\circ_{q,p} a_{p,t}  \circ c \rangle. $$ The sum here runs over   all  $ p\in a\cap \alpha_k$, $ q\in b\cap \alpha_k$, $ t\in c\cap \alpha_k$ satisfying the indicated inequalities which reformulate  the conditions $(k,p,q)\in T_\omega(a,b)$ and $(l,s,t)\in T_\omega(a,c)$. 
     The description of the loop $  b\circ_{q,p} a_{p,t}  \circ c  $ above shows that it is freely homotopic to  $  a_p b_q c_t   $ if $\varepsilon_p =1$ and 
  to $c_t b_q a_p$ if $\varepsilon_p =-1$. Thus, 
    $$\langle b\circ_{q,p} a_{p,t}  \circ c \rangle=\eta (\varepsilon_p) \langle a_p  b_q   c_t  \rangle+ ( \eta (\varepsilon_p ) -\varepsilon_p) \langle  c_t  b_q a_p  \rangle   
    $$ and therefore 
     \begin{equation}\label{1a} 
    {\mathcal L}_{\omega,k}^{a,b,c}=  \sum_{  
     { q<_\omega p, \, t<_\omega p}
  } \,  \varepsilon_q \, \varepsilon_t    \, \big (\eta (\varepsilon_p) \langle a_p  b_q   c_t  \rangle+ (\eta (\varepsilon_p ) -\varepsilon_p) \langle  c_t  b_q a_p  \rangle   \big ). \end{equation}
Cyclically permuting     $a,b,c$ and    $p,q,t$, we get 
  \begin{equation}\label{3a} 
    {\mathcal L}_{\omega,k}^{b,c,a}=  \sum_{  
     { t<_\omega q, \, p<_\omega q}
  } \,  \varepsilon_t \, \varepsilon_p    \, \big (
    \eta (\varepsilon_q)  \langle a_p  b_q c_t \rangle + (\eta (\varepsilon_q)-\varepsilon_q) \langle c_t b_q a_p  \rangle \big ) ,\end{equation}
     \begin{equation}\label{5a} 
    {\mathcal L}_{\omega,k}^{c,a,b}=  \sum_{  
     {p<_\omega t, \,q<_\omega t}
  } \,  \varepsilon_p\, \varepsilon_q    \, \big ( 
    \eta (\varepsilon_t)  \langle a_p  b_q c_t  \rangle +(\eta (\varepsilon_t)-\varepsilon_t) \langle c_t b_q a_p  \rangle \big ) . \end{equation}
Similarly, the  contribution to $(x\bullet y) \bullet z$ of  the right 3-terms  with the given~$k$
is   equal to
  $${\mathcal P}_{\omega,k}^{a,b,c}=    \sum_{   t<_\omega  q <_\omega p}
    \,  \varepsilon_p   \, \varepsilon_t   \,  \langle a\circ_{p,q} b_{q,t}  \circ c \rangle   $$
    where  the conditions on $p,q=s,t$  reformulate  the inclusions $(k,p,q)\in T_\omega(a,b)$ and $(l,s,t)\in T_\omega(b,c)$.  Since
    $$ \langle a\circ_{p,q} b_{q,t}  \circ c \rangle= \eta (\varepsilon_q)  \langle  c_t b_q  a_p \rangle + 
   (\eta (\varepsilon_q)-\varepsilon_q) \langle a_p  b_q c_t \rangle  )$$ we have  \begin{equation}\label{2a} 
    {\mathcal P}_{\omega,k}^{a,b,c}=    \sum_{   t<_\omega  q <_\omega p}
    \,  \varepsilon_p   \, \varepsilon_t   \, \big ( \eta (\varepsilon_q)  \langle  c_t b_q  a_p \rangle + 
   (\eta (\varepsilon_q)-\varepsilon_q) \langle a_p  b_q c_t \rangle  \big ). \end{equation}
    Cyclically permuting     $a,b,c$ and    $p,q,t$, we get 
 \begin{equation}\label{4a} 
    {\mathcal P}_{\omega,k}^{b,c,a}=    \sum_{  p<_\omega  t <_\omega q}
    \,  \varepsilon_q   \, \varepsilon_p   \, \big ( \eta (\varepsilon_t) \langle c_t b_q a_p  \rangle + 
   (\eta (\varepsilon_t)-\varepsilon_t)  \langle  a_p  b_q c_t  \rangle \big ),  \end{equation}
     \begin{equation}\label{6a} 
    {\mathcal P}_{\omega,k}^{c,a,b}=    \sum_{  q<_\omega p<_\omega t}
    \,  \varepsilon_t   \, \varepsilon_q   \, \big ( \eta (\varepsilon_p) \langle c_t b_q a_p  \rangle + 
    (\eta (\varepsilon_p)-\varepsilon_p)   \langle a_p  b_q c_t  \rangle \big ). \end{equation}
    The  contribution of the 3-terms  (with given~$k$) to $u_\omega (x,y,z)$ is   the sum of the expressions  \eqref{1a}--\eqref{6a}. Then     $P_{3,k} (a,b,c)$  is the sum of these  six expressions  and  of similar expressions  obtained  by replacing~$\omega$ with~$\overline{\omega}$.  Under this replacement, the only change  on the right-hand sides of  \eqref{1a}--\eqref{6a} concerns  the summation domain. For example,  the summation domain in \eqref{1a}  changes from 
    the set of triples $p,q,t$ such that  ${ q<_\omega p, \, t<_\omega p}$ to the set of triples $p,q,t$ such that ${ q<_{\overline{\omega}} p, \, t<_{\overline{\omega}} p}$. The latter conditions may be rewritten as  ${ p<_{ {\omega}} q, \, p<_{ {\omega}} t}$.  Thus,
     \begin{equation}\label{1az} 
    {\mathcal L}_{\overline{\omega},k}^{a,b,c}=  \sum_{  
     { p<_\omega q, \, p<_\omega t}
  } \,  \varepsilon_q \, \varepsilon_t    \, \big (\eta (\varepsilon_p) \langle a_p  b_q   c_t  \rangle+ (\eta (\varepsilon_p ) -\varepsilon_p) \langle  c_t  b_q a_p  \rangle   \big ) ,\end{equation}
  \begin{equation}\label{3az} 
    {\mathcal L}_{\overline{\omega},k}^{b,c,a}=  \sum_{  
     { q<_\omega t, \, q<_\omega p}
  } \,  \varepsilon_t \, \varepsilon_p    \, \big (
    \eta (\varepsilon_q)  \langle a_p  b_q c_t \rangle + (\eta (\varepsilon_q)-\varepsilon_q) \langle c_t b_q a_p  \rangle \big ) ,\end{equation}
     \begin{equation}\label{5az} 
    {\mathcal L}_{\overline{\omega},k}^{c,a,b}=  \sum_{  
     {t<_\omega p, \,t<_\omega q}
  } \,  \varepsilon_p\, \varepsilon_q    \, \big ( 
    \eta (\varepsilon_t)  \langle a_p  b_q c_t  \rangle +(\eta (\varepsilon_t)-\varepsilon_t) \langle c_t b_q a_p  \rangle \big ) ,\end{equation}
 \begin{equation}\label{2az} 
    {\mathcal P}_{\overline{\omega},k}^{a,b,c}=    \sum_{    p<_\omega q <_\omega  t}
    \,  \varepsilon_p   \, \varepsilon_t   \, \big ( \eta (\varepsilon_q)  \langle  c_t b_q  a_p \rangle + 
   (\eta (\varepsilon_q)-\varepsilon_q) \langle a_p  b_q c_t \rangle  \big ), \end{equation}
 \begin{equation}\label{4az} 
    {\mathcal P}_{\overline{\omega},k}^{b,c,a}=    \sum_{  q<_\omega  t <_\omega p}
    \,  \varepsilon_q   \, \varepsilon_p   \, \big ( \eta (\varepsilon_t) \langle c_t b_q a_p  \rangle + 
   (\eta (\varepsilon_t)-\varepsilon_t)  \langle  a_p  b_q c_t  \rangle \big ),  \end{equation}
     \begin{equation}\label{6az} 
    {\mathcal P}_{\overline{\omega},k}^{c,a,b}=    \sum_{  t<_\omega p<_\omega q}
    \,  \varepsilon_t   \, \varepsilon_q   \, \big ( \eta (\varepsilon_p) \langle c_t b_q a_p  \rangle + 
    (\eta (\varepsilon_p)-\varepsilon_p)   \langle a_p  b_q c_t  \rangle \big ). \end{equation}

  For any triple $p\in a\cap \alpha_k, q \in b\cap \alpha_k , t \in c\cap \alpha_k$,     consider the expressions \eqref{1a} -\eqref{6az}   and pick their    $(p,q,t)$-summands (some of them   may be  equal to zero).    We claim that the sum of these  twelve summands is equal to $\begin{vmatrix}
a&b&c\\
p&q&t\\
\end{vmatrix} $. 
    This will imply~\eqref{00}. 
To prove our claim,    consider possible   positions of the points $p,q,t$ on $\alpha_k$. Replacing, if necessary, $\omega$ by~$\overline \omega$, we can assume that $p<_\omega q$.  This leaves us with three  cases:
(a)  $t<_\omega  p$; (b) $p<_\omega  t <_\omega  q$, and (c) $q <_\omega t$. In Case (a), only the $(p,q,t)$-summands of ${\mathcal L}^{b,c,a}_{\omega, k}$, ${\mathcal L}^{c,a,b}_{\overline \omega, k}$, 
${\mathcal P}^{c,a,b}_{\overline \omega, k}$ may be non-zero  and their sum is   
$$ \varepsilon_t \, \varepsilon_p   
    \eta (\varepsilon_q)  \langle a_p  b_q c_t \rangle + \varepsilon_t \, \varepsilon_p    (\eta (\varepsilon_q)-\varepsilon_q) \langle c_t b_q a_p  \rangle $$
    $$+ \varepsilon_p\, \varepsilon_q    
    \eta (\varepsilon_t)  \langle a_p  b_q c_t  \rangle + \varepsilon_p\, \varepsilon_q  (\eta (\varepsilon_t) -\varepsilon_t) \langle c_t b_q a_p  \rangle$$
    $$+\varepsilon_t   \, \varepsilon_q    \eta (\varepsilon_p) \langle c_t b_q a_p  \rangle + 
  \varepsilon_t   \, \varepsilon_q     (\eta (\varepsilon_p)-\varepsilon_p)   \langle a_p  b_q c_t  \rangle  
  =\begin{vmatrix}
a&b&c\\
p&q&t\\
\end{vmatrix} .$$  
 In Case (b), only the $(p,q,t)$-summands of ${\mathcal L}^{b,c,a}_{\omega, k}$,  ${\mathcal P}^{b,c,a}_{\omega, k}$,  ${\mathcal L}^{a,b,c}_{\overline \omega, k}$  may be non-zero and their sum is
easily computed to be equal to $ \begin{vmatrix}
a&b&c\\
p&q&t\\
\end{vmatrix}$.  
  In Case (c), only the $(p,q,t)$-summands of ${\mathcal L}^{c,a,b}_{\omega, k}$,  ${\mathcal L}^{a,b,c}_{\overline \omega, k}$, ${\mathcal P}^{a,b,c}_{\overline \omega, k}$ may be non-zero and again their sum is  equal to $ \begin{vmatrix}
a&b&c\\
p&q&t\\
\end{vmatrix}$.  
  This proves the claim above  and completes the proof of the theorem.

 \section{Proof of Theorem~\ref{MAIN2} }\label{section3ghgh22}

\subsection{Conventions}\label{MAIN2P}  Set  $\nu=\nu_X: M\to M \otimes M$ and fix the gate orientation~$\omega$ of~$X$ such that $\varepsilon(\omega, k)=+1$ for all $k\in \pi_0$. 
 For points $p,q$ of a  gate, we write $p< q$ if $p<_\omega q$, i.e., if the $\omega$-positive direction on this gate leads from~$p$ to~$q$. 
Recall that a loop   in~$X$   is   \emph{simple} if   it is transversal to the gates  and has no self-intersections in~$\Sigma' \subset X$. By   Lemma~\ref{Loopsiemnn},  any loop  in~$X$  is   freely    homotopic to a simple   loop. 
  By Lemma~\ref{svvvvtreebucte},  for a  simple loop~$a$ in~$X$, we have 
\begin{equation}\label{mop-} 
\nu (\langle a \rangle)
=     \sum_{(k,p_1,p_2)\in T_\omega (a)}   \varepsilon_{p_1}  \, \varepsilon_{p_2}     \big ( \langle  a_{p_2,p_1} \rangle_0 \otimes \langle  a_{p_1,p_2} \rangle_0 - \langle  a_{p_1,p_2}\rangle_0 \otimes \langle  a_{p_2,p_1} \rangle_0  \big )  \end{equation} 
where    $\varepsilon_p=\varepsilon_p(a)$ for any point~$p$ of a gate traversed by~$a$. 
We say that the summand of \eqref{mop-}  determined by the chord $(k,p_1,p_2)\in T_\omega(a)  $ is   obtained by 
\emph{splitting}~$a$ at this chord.

\subsection{Proof of the theorem} 
  It is enough to prove that   \begin{equation}\label{coJ} (I+Q+Q^2)  \circ  \nu^2 (\langle a \rangle)= E  \circ  \gamma^3( \langle a \rangle) \end{equation} 
  for any simple loop~$a$ in~$X$. 
We  compute
 $
\nu (\langle a \rangle) $ via~\eqref{mop-}.
Note that the  loops $a_{p_1,p_2},  a_{p_2,p_1}$  in this formula   are not transversal to $\alpha_k$ as they contain   subpaths  connecting  $ p_1, p_2 \in a \cap \alpha_k$ in $\alpha_k$.
We deform both loops   pushing these subpaths     into $ X\setminus   \Sigma'$ while sliding their endpoints $ p_1, p_2$    into $ X\setminus   \Sigma'$ along $a(S^1)$.  This gives   simple loops   
$a'_{p_1,p_2},  a'_{p_2,p_1}$ in~$X$    whose intersections with the gates are among the intersections of~$a$ with the gates.  Therefore all chords of these two loops   are among  the chords of~$a$.   Since
$ \langle  a'_{p_1,p_2} \rangle_0 = \langle  a_{p_1,p_2} \rangle_0 $  and $ \langle  a'_{p_2,p_1} \rangle_0 = \langle  a_{p_2,p_1} \rangle_0$, we have  \begin{equation}\label{dd1} 
\nu (\langle a \rangle)  =\sum_{(k,p_1,p_2)\in T_\omega (a)}   \varepsilon_{p_1}   \varepsilon_{p_2}   \big ( \langle  a'_{p_2,p_1} \rangle_0 \otimes \langle  a'_{p_1,p_2} \rangle_0 - \langle  a'_{p_1,p_2}\rangle_0 \otimes \langle  a'_{p_2,p_1} \rangle_0  \big ) .   \end{equation}
Note    that for a chord   $(k,p_1,p_2)\in T_\omega (a)$,  the point $p_1 \in a\cap \alpha_k$ is traversed by the loop 
$a'_{p_1,p_2}$ if $\varepsilon_{p_1} =+1$ and  by the loop 
$a'_{p_2,p_1}$ if $\varepsilon_{p_1} =-1$. The point $p_2\in a\cap  \alpha_k$ is traversed by 
$a'_{p_2,p_1}$ if $\varepsilon_{p_2} =+1$ and  by  
$a'_{p_1,p_2}$ if $\varepsilon_{p_2} =-1$.


 It is clear that  $\nu^2 (\langle a \rangle)$ is a sum of expressions obtained by   splitting~$a$ along a   chord and then  splitting  one of the  resulting two    loops (deformed as in the previous paragraph)   along a    chord. The summation runs over  ordered pairs of  chords of~$a$ arising in this way. We call   such pairs   \emph{splitting pairs}.  We will say that two  chords   of~$a$ are \emph{unlinked} if (i) they have distinct endpoints  and  (ii)  going along the loop~$a$     we   meet consecutively both  endpoints of    one chord and then both endpoints of the other  chord.    It is clear  that the chords in a splitting pair are either unlinked or     have   at least one   common endpoint. 


We claim that the sum of the terms of $\nu^2 (\langle a \rangle)$ derived from  unlinked pairs of chords is annihilated by  $I+Q+Q^2$.
Indeed, pick any  unlinked pair  of   chords $(k,p_1,p_2), (l,u_1, u_2)  $ of~$a$  (possibly, $k=l$). Then  the pair   $ (l,u_1, u_2) , (k,p_1,p_2) $ also is unlinked. We will  show  that the terms  of $(I+Q+Q^2) \nu^2 (\langle a \rangle)$ derived from these two pairs   are opposite to  each other. This   will imply our claim.    Consider  first the case where the loop~$a$ traverses the points $p_1,p_2, u_1, u_2  
 $    in the  cyclic order  $p_2, p_1, u_2,  u_1$. Let $b ,  c, d, e$
  be respectively the paths in~$X$ going
  along~$a$ from $p_2$ to $p_1$, from $p_1$ to $u_2$,  from $u_2$ to $u_1$, and from $u_1$ 
     to $p_2$.   Let $f$ be the path going  from $p_1$ to $p_2$ along~$\alpha_k$ and let $g$ be the path going from $u_1$ 
     to $u_2$ along~$\alpha_l$.   
The  term of $\nu (\langle a \rangle)$ obtained    by   splitting~$a$ at   $(k,p_1,p_2)$  is  equal to
\begin{equation}\label{sle1}  \varepsilon_{p_1}   \varepsilon_{p_2}     
\big (   \langle b f \rangle_0  \otimes  \langle c d e f^{-1}\rangle_0    -  \langle c d e f^{-1} \rangle_0 \otimes   \langle b f\rangle_0 \big ).
\end{equation}
Consider the loop $v=e f^{-1} c g^{-1} $.
Splitting the left tensor factors in  \eqref{sle1}  at $(l,u_1, u_2)  $ we get
\begin{equation}\label{sle2}  \varepsilon_{p_1}  \varepsilon_{p_2}  \varepsilon_{u_1}    \varepsilon_{u_2}   
\big (  0 +  \langle  v \rangle_0 \otimes \langle d g \rangle_0  \otimes   \langle b f\rangle_0  -  \langle d g\rangle_0 \otimes \langle v \rangle_0  \otimes   \langle bf \rangle_0 \big )  
\end{equation} where the zero term    reflects the fact that $ (l,u_1, u_2)  $ is not a chord of the loop $bf$. 
Similarly,  the   summand of $\nu^2 (\langle a \rangle)$ arising   by splitting~$a$ first  at   $(l,u_1, u_2)$ and then  at $(k,p_1,p_2)  $ is  equal to 
\begin{equation}\label{sle3}  \varepsilon_{p_1}   \varepsilon_{p_2}    \varepsilon_{u_1}    \varepsilon_{u_2}   
\big (  \langle v \rangle_0 \otimes \langle b f \rangle_0  \otimes \langle d g \rangle_0 - \langle b f\rangle_0 \otimes \langle v \rangle_0  \otimes \langle dg \rangle_0    \ \big ) .\end{equation}    
Applying $I+Q+Q^2$ to  \eqref{sle2} and  \eqref{sle3},  we obtain opposite values. 
The cases  where  the  points    $p_1,p_2, u_1, u_2  
 $ are traversed by~$a$ in other possible cyclic orders     $(p_2,p_1,u_1,u_2)$,  $(p_1,p_2,u_1, u_2)$, and  $(p_1,p_2,u_2, u_1)$ are treated similarly.    


We check next that a splitting pair  of chords  consisting of the same   chord $(k,p_1,p_2)$ of~$a$ taken twice contributes zero to $ \nu^2 (\langle a \rangle)$. If $\varepsilon_{p_1}=
\varepsilon_{p_2}$ then one of the points $p_1, p_2$ is traversed by    $a'_{p_1,p_2}$ and the other one by  $a'_{p_2,p_1}$ so that   $(k,p_1,p_2)$  is not a chord of   these   loops. Then  the  pair   $(k,p_1,p_2), (k,p_1,p_2)$ is not a splitting pair of~$a$ and does not contribute to $\nu^2 (\langle a \rangle)$. If  $\varepsilon_{p_1}=
-\varepsilon_{p_2}$ then both  points $p_1,p_2$ are traversed by the same loop, either  $a'_{p_1,p_2}$ or $a'_{p_2,p_1}$.  Splitting that loop at the chord $(k,p_1,p_2)$, we obtain two loops   one of which lies in a neighborhood of $\alpha_k$ and  is contractible in~$X$. So, the  corresponding term of $\nu^2 (\langle a \rangle)$ is equal to  zero. 

It remains to consider     splitting pairs  $(k,p_1,p_2), (l,u_1, u_2)$ consisting of     chords of~$a$   with   one common endpoint. Then $k=l$ and our   chords    have 3 distinct endpoints. For     any three distinct points of $ a  \cap \alpha_k$ we  compute the contribution to $ \nu^2 (\langle a \rangle)$ of  all    associated splitting pairs. Label our  points  $p_1,p_2, p_3$  so that $p_1<_\omega p_2<_\omega p_3$. 
The splitting pairs   in question    are ordered pairs of distinct  chords from the list $(k,p_1,p_2), (k,p_2,p_3), (k,p_1,p_3)$. Recall the function~$\eta$ on the set $\{+1, -1\}$ defined by $\eta(+1)=1, \eta (-1)=0$.
 In the following computations,   $$x=
\langle a'_{p_1,p_2} \rangle_0 \in M , \,\, y = \langle a'_{p_2,p_3} \rangle_0 \in M , \,\, z=\langle a'_{p_3, p_1} \rangle_0 \in M. $$   Consider  first the case where the loop~$a$ traverses the points $p_1,p_2, p_3  
 $  in the  cyclic order  $p_1,p_2, p_3$.  In this  notation, the contribution of  the chord $(k,p_1,p_2) $   to $ \nu (\langle a \rangle)$ is equal to $$  \varepsilon_{p_1}    \varepsilon_{p_2}   \big ( \langle a'_{p_2,p_1} \rangle_0 \otimes x -x\otimes \langle a'_{p_2,p_1} \rangle_0 \big ) .$$   Clearly,  the loop $a'_{p_1,p_2}$ representing~$x$ does not traverse the point $p_3$ and  does not have   $    (k,p_2,p_3)$ as a chord. If  $\varepsilon_{p_2} =-1$, then the loop $a'_{p_2,p_1}$ does not traverse $p_2$ and does not have   $     (k,p_2,p_3)$ as a chord. 
Therefore the contribution of   the splitting pair   $(k,p_1,p_2), (k,p_2,p_3)$   to $ \nu^2 (\langle a \rangle)$  
is  equal to
$$\eta{(\varepsilon_{p_2}) } \varepsilon_{p_1}     \varepsilon_{p_2}   \varepsilon_{p_2}    \varepsilon_{p_3}   (z\otimes y -y\otimes z) \otimes x
=\eta{(\varepsilon_{p_2}) }  \varepsilon_{p_1}     \varepsilon_{p_2}    \varepsilon_{p_3}   (z\otimes y  -y\otimes z ) \otimes x. $$
Similarly,   the splitting pair   $(k,p_1,p_2), (k,p_1, p_3)$  contributes to $ \nu^2 (\langle a \rangle)$
the term 
$$ (1-\eta{(\varepsilon_{p_1}) }  )  \, \varepsilon_{p_1}  \varepsilon_{p_1}     \varepsilon_{p_2}      \varepsilon_{p_3}   (z\otimes y    -y\otimes z )\otimes x$$
$$= (\eta{(\varepsilon_{p_1}) }  -1 )  \, \varepsilon_{p_1}    \varepsilon_{p_2}      \varepsilon_{p_3}   (z\otimes y    -y\otimes z )\otimes x   .$$
So,   the  splitting pairs starting with    $(k,p_1,p_2)$ contribute to $ \nu^2 (\langle a \rangle)$ the expression 
$$(\eta{(\varepsilon_{p_1}) }  + \eta{(\varepsilon_{p_2}) }  -1 ) \,  \varepsilon_{p_1}     \varepsilon_{p_2}     \varepsilon_{p_3}   (z\otimes y  \otimes x -y\otimes z \otimes x). $$
Similarly,   the contributions  to $ \nu^2 (\langle a \rangle)$ of the splitting pairs with first term  $(k, p_2,p_3)$ and the second term $(k,p_1,p_2)$ or $(k,p_1, p_3)$ sum up  to 
$$(\eta{(\varepsilon_{p_2}) }  + \eta{(\varepsilon_{p_3}) }  -1 ) \, \, \varepsilon_{p_1}     \varepsilon_{p_2}     \varepsilon_{p_3}   (z\otimes x  \otimes y - x\otimes z \otimes y) $$
and the contributions  of the splitting pairs with first chord $(k,p_1,p_3)$ sum up  to 
$$(\eta{(\varepsilon_{p_1}) }  + 1- \eta{(\varepsilon_{p_3}) }   ) \, \, \varepsilon_{p_1}     \varepsilon_{p_2}      \varepsilon_{p_3}    (x\otimes y  \otimes z -y\otimes x \otimes z) .$$
Applying $I+Q+Q^2$ to the sum of these three expressions we get  
\begin{equation}\label{QQQ}  \varepsilon_{p_1}     \varepsilon_{p_2}    \varepsilon_{p_3} \, (I+Q+Q^2)  (  x\otimes y  \otimes z- 
z\otimes y  \otimes x ). \end{equation}
By  definition,  $\gamma^3 ( \langle a \rangle) \in M^{\otimes 3}$  is a sum whose terms are numerated by triples of distinct points of~$S^1$ carried by  $a:S^1\to X$ to  the same gate. The term associated with the triple $ p_1,p_2, p_3 $   above  is  $$ \varepsilon_{p_1}   \varepsilon_{p_2}     \varepsilon_{p_3}     (  x\otimes y  \otimes z+y \otimes z  \otimes x +z \otimes x  \otimes y). $$
Applying  the operator~$E$ (defined by  \eqref{HHH}) to  this expression we again get  \eqref{QQQ}. Thus, the  triple $p_1,p_2, p_3$   contributes  the same    to both sides of \eqref{coJ}. Similar computations   show  that this is also true   when the points $p_1,p_2, p_3 $ are traversed by~$a$  in  the  cyclic order  $p_3, p_2,p_1$. 
Summing up over all $k \in \pi_0$ and all triples  $p_1,p_2, p_3\in a \cap \alpha_k $  with    $ p_1 <  p_2 <  p_3$, we get~\eqref{coJ}. 

\section{Proof of Theorem~\ref{MAIN3}}\label{coorcompadfT0}

We start by computing  $\partial \nu (x \otimes y)$ 
  for arbitrary  $x,y \in  L(X) \subset M $. Recall  that 
  $$\partial \nu (x \otimes y)= \nu ( [x,y])  - ad_x( \nu(y)) +   ad_y( \nu(x)) .$$
   Since  all terms on the right-hand side involve    two of our  (co)bracket operations,  they expand  as    sums  of terms determined  by ordered pairs of chords. To make this expansion    explicit, we use notation  of Section~\ref{MAIN2P} and   represent $x,y$ by simple loops $a,b$    which do not meet   in~$\Sigma'$.  
For any chord $(k,p,q) \in T(a,b)$,  consider the simple loop $a\circ_{p,q}b$ in~$X$ obtained by deformation of the loop  $a_pb_q$ as   in the proof of 
 Theorem~\ref{MAIN} in  Section~\ref{section3ghvvgh}.   Formula~\eqref{monoee} and  the equality $ \langle a\circ_{p,q}b \rangle =  \langle b\circ_{q,p }a \rangle $  imply  that 
$$ [x,y] =x \bullet y -y \bullet x =  \sum_{(k,p,q) \in T_\omega (a,b) }\,    \varepsilon_p  \,\varepsilon_q \,   \langle a\circ_{p,q}b \rangle - \sum_{(k,q,p) \in T_\omega (b,a) }\,    \varepsilon_p  \,\varepsilon_q \,   \langle b\circ_{q,p }a \rangle $$
$$= 
 \sum_{(k,p,q) \in T(a,b) }\,   \delta (p,q)\,  \varepsilon_p  \,\varepsilon_q \,   \langle a\circ_{p,q}b \rangle $$
where  $\varepsilon_p=\varepsilon_p(a)$, $\varepsilon_q=\varepsilon_q(b)$, and
\[  \delta  (p,q) =  \begin{cases}
    +1   & \quad \text{if } q< p,  \\
    -1    & \quad \text{if }  p < q.
  \end{cases} \]
   Therefore 
$$\nu ( [x,y] )= 
 \sum_{(k,p,q) \in T(a,b) }\,   \delta (p,q)\,  \varepsilon_p   \,\varepsilon_q \,  \nu (  \langle a\circ_{p,q}b  \rangle ) .$$
By  \eqref{mop-},   $ \nu (  \langle a\circ_{p,q}b \rangle )$ is a sum of terms   numerated by the   chords of  $a\circ_{p,q}b$. 
Thus,  $\nu ( [x,y])$ is a sum of  terms   numerated by pairs (a chord $(k,p,q)$ of the pair   $(a,b)$, a   chord  of the loop    $a\circ_{p,q}b$).   We call a chord $(k,p,q)$ of the pair  $(a,b)$ \emph{positive} if $p<q$. For all $l\in \pi_0$, we have    \begin{equation}\label{MMT} (a\circ_{p,q}b) \cap \alpha_l = (a \cap \alpha_l) \sqcup (b \cap \alpha_l). \end{equation}    Therefore a  chord of the loop $a\circ_{p,q}b$ is  either  a  chord of~$a$, or a  chord   of~$b$, or a positive chord of the
 pair $(a,b)$, or a positive chord of the pair~$(b,a)$.  Next we use Formula~\eqref{dd1} to compute $\nu(x)$.  Since the crossings of the gates with the loops $ a'_{p_1,p_2},  a'_{p_2,p_1}$ appearing in~\eqref{dd1}    are among the crossings of the gates with~$a$,  the   brackets   $[y,   \langle  a'_{p_1,p_2}\rangle_0] $ and  $  [y,   \langle  a'_{p_2, p_1}\rangle_0 ]$  are  sums of expressions  numerated by    chords   of  the pair   $(b,a)$.
 So,  $ad_y( \nu(x))$ is a sum  of  terms   numerated by pairs (a  chord   of~$a$, a chord  of  $(b,a)$). Similarly,  $ad_x( \nu(y))$ is a sum  of terms  numerated by pairs (a  chord   of~$b$, a chord  of  $(a,b)$).
We conclude that  there are five types of ordered pairs of   chords  which may contribute to $\partial \nu (x \otimes y)$:

(i)   (a chord  of  $(a,b)$, a  chord  of~$a$); 

(ii)  (a chord  of  $(a,b)$, a  chord  of~$b$); 

(iii)  (a chord  of  $(a,b)$, a positive chord  of  $(a,b)$ or    of  $(b,a)$);

(iv)   (a  chord   of~$a$, a chord  of  $(b,a)$);

(v) (a  chord   of~$b$, a chord  of  $(a,b)$). 

Here pairs of types (i)--(iii) contribute to  $\nu ( [x,y])$ and  pairs of types (iv), (v) contribute respectively  to $ad_y( \nu(x))$, $ad_x( \nu(y)) $.  In the sequel, the contribution of an ordered pair of chords~$s$ to  $\partial \nu (x \otimes y)$ is denoted by$\vert s\vert$.


Pick  now a   set $S\subset \cup_k \alpha_k $ such that there is an ordered
  pair of  chords in the list (i)--(v)   whose set of endpoints  
  is equal to~$S$.  We call such a pair   of chords an \emph{$S$-pair}. The existence of an  $S$-pair implies that   $2 \leq \card(S) \leq 4$
  and~$S$ meets both $a(S^1)$  and $ b(S^1)$.   Set $\vert S\vert=\sum_s \vert s\vert$ where~$s$ runs over all $S$-pairs.

Claim:    $\card(S)=4 \Rightarrow \vert S\vert=0$. 
We first consider the case where~$S$ meets both $a(S^1)$  and $b(S^1)$ in two points. Then all  $S$-pairs have type (iii) and~$S$ consists  of  four distinct   points   $$p\in a\cap \alpha_k , \quad  q\in b\cap \alpha_k , \quad p'\in a\cap \alpha_l ,  \quad q'\in b\cap \alpha_l$$  with    $k,l \in \pi_0 $ (possibly, $k=l$).  
We  let $c^{p,q}_{p',q' } $ be the loop   going from $p'$ to $q'$ along $a\circ_{p,q}b$ and then   back to $p'$ along  the gate $\alpha_l$.
Viewed up to homotopy, this loop     goes  along~$a$ from~$p'$ to~$p$,  then along $\alpha_k$ to $q$, then along~$b$ to $q'$,  and then   along  $\alpha_l$   back to~$p'$.  Similarly, we let $c^{p,q}_{q',p' } $ be the loop   going from  $q'$ to $p'$ along $a\circ_{p,q}b$ and then  back to~$q'$  along  $\alpha_l$.   Viewed up to homotopy, this  loop     goes  along~$b$ from $q'$ to~$q$,  then along $\alpha_k$ to~$p$, then along~$a$ to~$p'$,  and then   along  $\alpha_l$   back to~$q'$.

Note that there is just one $S$-pair,~$s$, with the first term $(k,p,q)$. Namely,
 $$  s =  \begin{cases}
     ((k,p,q), (l,p',q'))    & \quad \text{if }  p'< q', \\
     ((k,p,q), (l,q',p'))   & \quad \text{if } q'< p'.
  \end{cases}    $$
 Clearly,   $$ \vert s \vert =\delta (p,q) \,   \delta (p',q') \,  \varepsilon_p  \,  \varepsilon_q \,\varepsilon_{p'}   \, \varepsilon_{q'}  \, \big (\langle  c^{p,q}_{p',q'} \rangle_0 \otimes  \langle c^{p,q}_{q',p'} \rangle_0 - \langle c^{p,q}_{q',p'} \rangle_0 \otimes \langle  c^{p,q}_{p',q'} \rangle_0 \big) $$
where   $ \varepsilon_p = \varepsilon_p(a), \varepsilon_{p'}=\varepsilon_{p'}(a) , \varepsilon_q =\varepsilon_q (b),    \varepsilon_{q'}  =\varepsilon_{q'} (b)$. 
We    rewrite  this   as
$$ \vert s \vert =\delta (p,q) \,   \delta (p',q') \,  \varepsilon_p  \,  \varepsilon_q \,\varepsilon_{p'}   \, \varepsilon_{q'}  \, \overline P(\langle  c^{p,q}_{p',q'} \rangle_0 \otimes  \langle c^{p,q}_{q',p'} \rangle_0). $$
 Similarly,  there is a unique $S$-pair  $t$ with the first term $(l,p', q')$   and 
$$  \vert t \vert = \delta (p,q) \,   \delta (p',q') \,  \varepsilon_p \,  \varepsilon_q\,\varepsilon_{p'}  \, \varepsilon_{q'}   \,  \overline P(\langle  c^{p',q'}_{p,q} \rangle_0 \otimes  \langle c^{p',q'}_{q,p} \rangle_0) . $$
The obvious equalities    $\langle c^{p',q'}_{q,p}  \rangle_0 =\langle  c^{p,q}_{p',q'} \rangle_0 $  and
$\langle c^{p',q'}_{p,q}  \rangle_0 =\langle  c^{p,q}_{q',p'} \rangle_0 $ imply  that 
 $\vert s \vert=-\vert t \vert$. If $k\neq l$, then there are no other $S$-pairs and    $\vert S\vert=\vert s \vert+\vert t \vert=0$. If $k=l$, then there are two more $S$-pairs~$s'$ and~$t'$  
  with the first terms respectively $(k,p,q')$ and   $(k, p', q)$. The argument above (with~$p$ replaced by~$p'$ and vice versa)  shows that $\vert s'\vert=-\vert t' \vert$.  So,  $\vert S\vert=\vert s \vert+\vert t \vert +\vert s'\vert+\vert t' \vert =0$.
  


Suppose  now  that our  4-element set~$S$ meets $a(S^1)$ in three points and meets $b(S^1)$ in one point. The existence of an $S$-pair implies that~$S$ consists  of    points   $$p\in a \cap \alpha_k  , \quad q \in b \cap\alpha_k, \quad p', p''\in  a\cap \alpha_l$$  for some    $k,l \in \pi_0$.
 Assume first that $k \neq l$   and   choose  $p', p''$ so  that  $p'<p''$.
 Then there are only  two $S$-pairs of chords: \begin{equation}\label{nnnx-} s=((k,p,q),  (l,p',p''))\quad \text{and} \quad t=( (l,p',p''), (k,q,p)) \end{equation} of  types respectively (i) and    (iv).   
If  the points $p,p', p''$ are traversed by~$a$ in this cyclic  order, then $$\vert s \vert=    \delta (p,q)\,  \varepsilon_p  \,  \varepsilon_q \,  \varepsilon_{p'}  \,  \varepsilon_{p''} \, 
\overline P \big (  \langle (a_{p'',p'})_p b_q\rangle_0  \otimes   \langle a_{p',p''} \rangle_0 \big ) . $$
If  the points $p,p', p''$ are traversed by~$a$  in the opposite  cyclic  order, then
$$\vert s \vert=   \delta (p,q)\,  \varepsilon_p  \,  \varepsilon_q  \,  \varepsilon_{p'}  \,  \varepsilon_{p''} \, \overline P
\big (  \langle a_{p'',p'}\rangle_0  \otimes   \langle (a_{p',p''})_p  b_q \rangle_0   \big ) . $$
   At the same time, the contribution of the  chord  $ (l,p',p'')$ of~$a$  to $\nu(x)$ is equal to
$$  \varepsilon_{p'} \,  \varepsilon_{p''}    \,   \overline P
\big ( \langle a_{p'',p'} \rangle_0  \otimes   \langle a_{p',p''} \rangle_0    \big ).$$
As a consequence, $\vert t \vert=-\vert s \vert $ where the minus  is  due to the minus in the formula   $\delta(q,p) =-\delta (p,q)$.   Thus,  $\vert S\vert=\vert s \vert+\vert t \vert=0$.  If $k=l$, then we label the three points    of  the set  $S\cap a(S^1)$ by $p,p', p''$  so that  $p<p'<p''$. Here, besides the $S$-pairs $s,t$ as in  \eqref{nnnx-}, we have four more $S$-pairs
$$  s'=((k,p',q),  (l,p,p'')),\quad  \quad t'=( (l,p,p''), (k,q,p')), $$
$$ s''=((k,p'',q),  (l,p,p')), \quad   \quad t''=( (l,p,p'), (k,q,p'')) .  $$
The same computations as above yield $$\vert s \vert+\vert t \vert= \vert s'
 \vert+\vert t' \vert= \vert s'' \vert+\vert t'' \vert=0.$$
Hence, $$\vert S\vert =\vert s \vert+\vert t \vert+ \vert s'
 \vert+\vert t' \vert+ \vert s'' \vert+\vert t'' \vert=0.$$

The case where~$S$ meets $a(S^1)$ in one point  and meets $b(S^1)$ in  three points is treated similarly.  This completes the proof of our    claim     $\card(S)=4 \Rightarrow \vert S\vert=0$.

Consider now the case   $\card(S)=2$. Then~$S$ consists of the endpoints of a  chord  $(k,p,q)$ of $(a,b)$. Suppose first that $p<q$. Then  the only $S$-pair    is    $s=( (k,p,q),  (k,p,q)) $. The   chord $ (k,p,q)$     contributes  to $[x,y] $ the expression 
$$ \delta (p,q)  \, \varepsilon_p \,  \varepsilon_q\,   \langle a \circ_{p,q}b \rangle=-\varepsilon_p \,  \varepsilon_q\,   \langle a \circ_{p,q}b \rangle .$$ The   same chord $(k,p,q)$ is   a chord of  the loop $a \circ_{p,q}b$  and splitting $a \circ_{p,q}b$  at  it  we obtain  two   loops.    If  $\varepsilon_p=-  \varepsilon_q$, then    one of them is contractible and therefore
 $\vert S\vert =\vert s\vert  =0$.
If $\varepsilon_p=  \varepsilon_q $, then   these two   loops   are freely homotopic to $a, b$.    Analyzing  separately the cases  $\varepsilon_p=  \varepsilon_q  =+ 1$ and $\varepsilon_p=  \varepsilon_q =-1$  we obtain that   $$  \vert S\vert =\vert s\vert =  \varepsilon_p\,  \overline P   (\langle a \rangle_0 \otimes \langle b \rangle_0 )  =   
  \varepsilon_p\, \overline P  (x \otimes y) \,\, \,\, {\text {mod}}  (Re \otimes M +M\otimes Re).   
$$ Thus, modulo $Re \otimes M +M\otimes Re$, we have
    \begin{equation} \label{2-element--}   \vert S\vert  =  \begin{cases}
0   & \quad \text{if } \varepsilon_p=-  \varepsilon_q, \\
  \varepsilon_p\, \overline P  (x \otimes y)   & \quad \text{if } \varepsilon_p=  \varepsilon_q.
  \end{cases}    \end{equation} 
    In the case $q<p$ there is  only one $S$-pair       $ ( (k,p,q),  (k,q,p)) $ and similar computations again give  \eqref{2-element--}.

Finally, assume that   $\card(S)=3$, i.e., that~$S$ consists of the endpoints of two distinct  chords sharing one endpoint. Then  $S$  is contained  in a single gate $\alpha_k$   and meets one of the sets $a(S^1), b(S^1) $ in two points and the other  set  in one point. Suppose first that
$S=\{p,p', q\} $ for some $p,p' \in  a \cap \alpha_k   $, $q \in  b\cap \alpha_k $ with $p < p'$. Consider the  loops $u=a_{p,p'}$,
$u'=a_{p',p}$,  $v=b_q$ based in  $p, p', q$ respectively. We  view  the gate $\alpha_k$ as a \lq\lq big'' base point of these loops; this    allows us to form  the  products    $uu', u'v$, etc.  
For example, in this notation $\langle a\circ_{p,q} b \rangle  =\langle uu'v\rangle $ and  $\langle a\circ_{p',q} b \rangle  =\langle u'uv\rangle $. Set
$$U=\langle  u \rangle  \otimes  \langle  u'v \rangle    = \langle  u \rangle \otimes  \langle  vu'\rangle  
\in M \otimes M$$
and 
$$U'=\langle u' \rangle  \otimes  \langle  uv \rangle   = \langle  u' \rangle  \otimes  \langle  vu\rangle 
\in M \otimes M.$$
We claim that  modulo $Re \otimes M +M\otimes Re$ we have  
\begin{equation}\label{mann} \vert S\vert =  \varepsilon_p \, \varepsilon_{p'} \,\varepsilon_q \,  {\overline P} ( U+U' ) . \end{equation} 
To see it,  we list all $S$-pairs of chords. Pairs of type (i):
$$ s= ((k,p,q), (k,p,p')) \, \,\, \, \text{and}\,\, \,\,    s'= ((k,p',q), (k,p,p')).$$
Pairs of types (ii) and (v): none.  Pairs of type (iii):
$$  t =  \begin{cases}
     ((k,p,q), (k, p',q))    & \quad \text{if }  p'< q, \\
     ((k,p,q), (k,q,p'))   & \quad \text{if } q< p',
  \end{cases}    $$ and
  $$  t' =  \begin{cases}
     ((k,p',q), (k, p,q))    & \quad \text{if }  p< q, \\
     ((k,p',q), (k,q,p))   & \quad \text{if } q< p.
  \end{cases}    $$
  Pairs of type (iv):
  $$w=((k,p,p'), (k, q,p))\, \,\, \, \text{and}\,\, \,\,  w'=((k,p,p'), (k, q,p')).$$
  We compute the contributions of   these   six  $S$-pairs to $\partial \nu (x\otimes y)$. All the computations to follow proceed  modulo $Re \otimes M +M\otimes Re$.  As we know,  the contribution of the chord  $(k,p,q)$   to $[x,y]$ is equal to 
\begin{equation}\label{as0e}  \delta (p,q)\,   \varepsilon_p\, \varepsilon_q \,   \langle a\circ_{p,q}b\rangle   .\end{equation} We apply  $\nu$  to \eqref{as0e} and focus  on the contribution of the chord $(k,p,p')$. This computes $\vert s\vert$ as follows:
$$ \vert s\vert = - \delta (p,q)\,   \varepsilon_p\,  \varepsilon_{p'} \, \varepsilon_q \,   \times  \begin{cases}
     {\overline P}   (  U'
    )    & \quad \text{if }  \varepsilon_p=-1, \\
        {\overline P}   (  U
    )     & \quad \text{if } \varepsilon_p=+1.
  \end{cases}    $$    Similarly,    the chord  $(k,p',q)$    contributes 
$ \delta (p',q)\,   \varepsilon_{p'}\, \varepsilon_q \,   \langle a\circ_{p',q}b \rangle $ to $[x,y]$.  Applying~$\nu$  and focusing on the contribution of the chord $(k,p,p')$, we get
$$ \vert s'\vert =   \delta (p',q)\,   \varepsilon_p\,  \varepsilon_{p'} \, \varepsilon_q \,   \times  \begin{cases}
     {\overline P}   (  U
    )    & \quad \text{if }  \varepsilon_{p'}=-1, \\
        {\overline P}   (  U'
    )     & \quad \text{if } \varepsilon_{p'}=+1.
  \end{cases}    $$   
 Next, applying $\nu$ to \eqref{as0e}  and focusing on the contribution of the chord $(k, p',q)  $ if $  p'<q$ and the chord $
  (k,q,p')$ if $ q< p'$,  we get
  $$ \vert t\vert =  \delta (p,q)\,  \delta (p',q)\,   \varepsilon_p\,  \varepsilon_{p'} \,   \varepsilon_q \,   \times  \begin{cases}
     {\overline P}   (  U
    )    & \quad \text{if }  \varepsilon_q=-1, \\
        {\overline P}   (  U'
    )     & \quad \text{if } \varepsilon_q=+1.
  \end{cases}    $$  
  This computation of $\vert t\vert$ does     not use the assumption $p< p'$; therefore exchanging $p ,U$   with~$p', U'$, respectively,  we get 
  $$ \vert t'\vert =  \delta (p,q)\,  \delta (p',q)\,   \varepsilon_p\,  \varepsilon_{p'} \,   \varepsilon_q \,   \times  \begin{cases}
     {\overline P}   (  U'
    )    & \quad \text{if }  \varepsilon_q=-1, \\
        {\overline P}   (  U
    )     & \quad \text{if } \varepsilon_q=+1.
  \end{cases}    $$
By definition,  the contribution of the chord  $(k,p,p')$   to $\nu(x)$ is equal to 
\begin{equation}\label{asmm0e}   \varepsilon_p\, \varepsilon_{p'} \,     (\langle u' \rangle_0  \otimes \langle u \rangle_0-\langle u \rangle_0  \otimes \langle u' \rangle_0 ).\end{equation}
Applying ${ {ad}}_y$ and focusing on the contribution of the chord $(k,q,p)$, we get 
  $$ \vert w\vert = - \delta (p,q)\,      \varepsilon_p\,  \varepsilon_{p'} \,   \varepsilon_q \,   \times  \begin{cases}
     {\overline P}   (  U
    )    & \quad \text{if }  \varepsilon_p=-1, \\
        {\overline P}   (  U'
    )     & \quad \text{if } \varepsilon_p=+1 
  \end{cases}    $$
  where we use the identity $\delta(q,p)=- \delta (p,q)$. 
  Applying ${ {ad}}_y$ to \eqref{asmm0e} and focusing on the contribution of the chord $(k,q,p')$, we similarly get 
  $$ \vert w'\vert =   \delta (p',q)\,      \varepsilon_p\,  \varepsilon_{p'} \,   \varepsilon_q \,   \times  \begin{cases}
     {\overline P}   (  U'
    )    & \quad \text{if }  \varepsilon_{p'}=-1, \\
        {\overline P}   (  U
    )     & \quad \text{if } \varepsilon_{p'}=+1.
  \end{cases}    $$
  From these formulas we deduce that
\begin{equation}\label{el11} \vert s \vert +\vert w \vert = - \delta (p,q)\,      \varepsilon_p\,  \varepsilon_{p'} \,   \varepsilon_q \,
  {\overline P}   (  U +U'
    ),\end{equation}
\begin{equation}\label{el12}\vert s' \vert +\vert w' \vert = \delta (p',q)\,      \varepsilon_p\,  \varepsilon_{p'} \,   \varepsilon_q \,
  {\overline P}   (  U +U'
    ),\end{equation}
\begin{equation}\label{el13}\vert t \vert +\vert t' \vert = \delta (p,q)\, \delta (p',q)\,      \varepsilon_p\,  \varepsilon_{p'} \,   \varepsilon_q \,
  {\overline P}   (  U +U'
    ). \end{equation}
  Note that
    $$- \delta (p,q) +  \delta (p',q) + \delta (p,q)\, \delta (p',q) =1$$
    as either $q<p'$ and then $ \delta (p',q)=1$ or $p<p'<q$ and then $\delta (p,q) =-1$.
    Therefore   adding up  the equalities \eqref{el11}-\eqref{el13}
    we obtain  \eqref{mann}. 
    
    It remains to consider the case where   
$S=\{p,q,q'\} $ for some $p \in  a \cap \alpha_k   $, $q,q' \in  b\cap \alpha_k $ with $q < q'$. Consider the  loops $u=a_p$,  $v=b_{q,q'}$,
$v'=b_{q',q}$ based respectively at $p, q,q'$. As above, we  view  the gate $\alpha_k$ as a \lq\lq big" base point of these loops.  
 Set
$$V=\langle v \rangle \otimes  \langle  v'u \rangle    = \langle  v \rangle  \otimes  \langle  uv'\rangle  
\in M \otimes M$$
and 
$$V'=\langle v' \rangle  \otimes  \langle  vu\rangle    = \langle  v' \rangle  \otimes  \langle  uv\rangle  
\in M \otimes M.$$
Then modulo $Re \otimes M +M\otimes Re$ we have
\begin{equation}\label{mann+} \vert S\vert = -  \varepsilon_p \,\varepsilon_q  \, \varepsilon_{q'} \,  {\overline P} ( V+V' ) . \end{equation} 
This follows from  the observation that the equality $\partial \nu(x \otimes y)=-\partial \nu(y \otimes x)$ holds  term-wise so that the contribution    of~$S$ to $\partial \nu(x \otimes y)$ is obtained by negation from   the contribution of~$S$ to $\partial \nu(y \otimes x)$. By \eqref{mann}, the latter    is  equal to $\varepsilon_p \,\varepsilon_q  \, \varepsilon_{q'} \,  {\overline P} ( V+V' ) $. This  implies \eqref{mann+}.

Let $\sigma_0$ be  the total contribution to $\partial \nu(x \otimes y)$ of the
  2-element  sets~$S$ consisting of the endpoints of a chord of $(a,b)$.  Summing up the equalities \eqref{2-element--} over all such~$S$ we obtain  that
$$ \sigma_0=  \sum_{k \in \pi_0}\,  (n_k^+-n_k^-)\,  \overline P    (   x \otimes y   ) $$ 
where   $n_k^\pm$ is the number of pairs $(p\in a\cap \alpha_k, q\in b\cap \alpha_k)$ with  $ \varepsilon_p(a)=\varepsilon_q(b)=\pm$. 
In the notation of Section~\ref{Gate bi-endos} applied  to $C=\alpha_k$  we have 
$$n_k^+-n_k^-=\frac{\vert a \vert_{\alpha_k} (b \cdot \alpha_k) +\vert b \vert_{\alpha_k} (a \cdot \alpha_k)}{2}.$$ Therefore  
$$2 \sigma_0=  \sum_{k \in \pi_0}\, (\vert a \vert_{\alpha_k} (b \cdot \alpha_k) +\vert b \vert_{\alpha_k} (a \cdot \alpha_k))\,  \overline P    (   x \otimes y   ) =\sigma'_0 - \sigma''_0$$
where 
$$\sigma'_0= \sum_{k \in \pi_0}\, \vert a \vert_{\alpha_k} (b \cdot \alpha_k)  \,  \overline P    (   x \otimes y   )\quad {\text {and}} \quad 
\sigma''_0= \sum_{k \in \pi_0}\, \vert b \vert_{\alpha_k} (a \cdot \alpha_k)  \,  \overline P    (   y\otimes x   ) . $$
Let $\sigma_1$ be  the total contribution to $\partial \nu(x \otimes y)$ of all sets~$S$  consisting of two  points   of $a  \cap \alpha_k$ and a point of $b  \cap \alpha_k$ for some $k \in \pi_0$.
Summing up    the equalities~\eqref{mann} over all such $S,k$, we obtain that (modulo $Re \otimes M +M\otimes Re$) 
$$\sigma_1= \sum_{k \in \pi_0} \,   \sum_{p_1, p_2 \in a  \cap \alpha_k, p_1\neq p_2} \,  \sum_{q\in  b  \cap \alpha_k }  \, \varepsilon_{p_1} (a) \varepsilon_{p_2} (a)  \varepsilon_{q} (b)  \, \overline P \big ( \langle a_{p_1, p_2}\rangle \otimes \langle a_{p_2, p_1} \circ b_q\rangle \big ).$$ Similarly,   Formula~\eqref{mann+} implies that the total contribution to $\partial \nu(x \otimes y)$ of all sets~$S$  consisting of a  point   of $a  \cap \alpha_k$ and  two points of $b  \cap \alpha_k$ for some~$k$ is equal   (modulo $Re \otimes M +M\otimes Re$) to $-\sigma_2$ where
$$\sigma_2=   \sum_{k \in \pi_0} \,   \sum_{p  \in a  \cap \alpha_k } \,  \sum_{q_1, q_2 \in  b  \cap \alpha_k, q_1\neq q_2 }  \, \varepsilon_{p} (a) \varepsilon_{q_1} (b)  \varepsilon_{q_2} (b)  \,  \overline P \big ( \langle b_{q_1,q_2}\rangle \otimes \langle b_{q_2, q_1} \circ a_p\rangle \big ). $$
Since $  \vert S\vert=0$ for all other~$S$, we have 
$$\partial \nu(x \otimes y) = \sigma_0+\sigma_1-\sigma_2 \mod (Re \otimes M +M\otimes Re).$$
Therefore
$$2\, \partial \nu(x \otimes y) = (\sigma'_0+2\sigma_1)-(\sigma''_0 +2\sigma_2) \mod (Re \otimes M +M\otimes Re).$$  In terms of the bi-endomorphism $\zeta$ of~$M$, we have $$\sigma'_0+2\sigma_1= \overline P\zeta (x\otimes y) \quad {\text{and}} \quad \sigma''_0 +2\sigma_2=\overline P\zeta (y\otimes x).$$ Thus, modulo $Re \otimes M +M\otimes Re$, we have 
$$2\, \partial \nu(x \otimes y)  =\overline P\zeta (x\otimes y -y\otimes x)= \overline P \zeta \overline P (x\otimes y)=\delta(\zeta)( x \otimes y). $$ So, $\delta(\zeta)=2\, \partial \nu $ modulo $Re \otimes M +M\otimes Re$.


\begin{thebibliography}{CJKLS}


 



\bibitem[AB]{AB} M. F. Atiyah, R. Bott,  \emph{The Yang-Mills equations over Riemann surfaces,} Philos.
Trans. R. Soc. Lond. Ser. A 308,  (1983), 523--615.

 

\bibitem[Ch]{Ch}  M. Chas,  \emph{Combinatorial Lie bialgebras of curves on surfaces,} Topology 43 (2004), 543–568.





\bibitem[FR]{FR}
V. V. Fock, A. A. Rosly,
\emph{Poisson structure on moduli of flat connections on Riemann surfaces and the $r$-matrix.}
(Russian) Moscow Seminar in Math. Physics.
English translation: Amer. Math. Soc. Transl. Ser. 2, 191, 67--86 (1999).
 
 
\bibitem[Go1]{Go1}
W. M. Goldman,
\emph{ The symplectic nature of fundamental groups of surfaces.}
Adv. in Math. 54 (1984), no. 2, 200--225.

\bibitem[Go2]{Go2}
W. M. Goldman,
\emph{Invariant functions on Lie groups and Hamiltonian flows of surface group representations.}
Invent. Math. 85 (1986), no. 2, 263--302.

 


 


\bibitem[KK1]{KK}
N. Kawazumi, Y. Kuno,
\emph{The logarithms of Dehn twists.}
 Quantum Topol. 5 (2014), no. 3, 347--423.

 
\bibitem[KK2]{KK2}
N. Kawazumi, Y. Kuno,
\emph{Intersection of curves on surfaces and their applications to mapping class groups.}
Ann. Inst. Fourier 65 (2015), no. 6, 2711--2762. 








\bibitem[MT]{MTnew}
G. Massuyeau, V. Turaev,
\emph{Quasi-Poisson Structures on Representation Spaces of Surfaces.}
International Math.  Research Notices (2014), 1--64



\bibitem[Tu1]{Tu2}
V.  Turaev,
\emph{Skein quantization of Poisson algebras of loops on surfaces.}
Ann. Sci. \'Ecole Norm. Sup. (4) 24 (1991), no. 6, 635--704.


\bibitem[Tu2]{Tu3}
V.  Turaev,
\emph{Loops in surfaces and star-fillings.} 
arXiv:1910.01602.

\bibitem[Wo]{Wo}
S. Wolpert,
\emph{On the symplectic geometry of deformations of a
hyperbolic surface.} Ann. of Math. (2) 117 (1983), 207--234.





                     \end{thebibliography}
    \end{document}